\documentclass[10pt,reqno]{amsart}
\usepackage{amsfonts,amsmath,amssymb,amsbsy,amstext,amsthm}
\usepackage{accents,color,enumerate,enumitem,float,fullpage,parskip,verbatim}

\usepackage{url}
\usepackage[colorlinks=true,hyperindex, linkcolor=magenta, pagebackref=false, citecolor=cyan]{hyperref}
\usepackage[alphabetic,lite,backrefs]{amsrefs} 

\usepackage{eucal,bm,kpfonts,mathbbol}

\usepackage{tikz,tikz-cd}	
\usetikzlibrary{positioning, matrix, shapes}         								    				
\usetikzlibrary{arrows,calc,matrix}

\usepackage{lscape}

\usepackage{titlesec}		
\titleformat{\section}[block]{\large\scshape\bfseries\filcenter}{\thesection.}{1em}{}		% Creates section titles
\titleformat{\subsection}[hang]{\large\scshape\bfseries}{\thesubsection}{1em}{}			% Creates subsection titles
\titleformat{\subsubsection}[hang]{\large\scshape\bfseries}{\thesubsubsection}{1em}{}			% Creates subsection titles

\usepackage[titles]{tocloft}								     					% Creates table of fancy contents
\usepackage{multirow}
\usepackage{array}
\usepackage{booktabs}
\newcolumntype{M}[1]{>{\centering\arraybackslash}m{#1}}
\newcolumntype{N}{@{}m{0pt}@{}}
\usepackage{diagbox}
\usepackage{cancel}

\newtheorem{lemma}{Lemma}[section]
\newtheorem{theorem}[lemma]{Theorem}
\newtheorem{prop}[lemma]{Proposition}
\newtheorem{cor}[lemma]{Corollary}
\newtheorem{conj}[lemma]{Conjecture}

\newtheorem{defn}[lemma]{Definition} 
\newtheorem{notation}[lemma]{Notation} 

\newtheorem{question}[lemma]{Question}

\newtheorem{theoremalpha}{Theorem}
\newtheorem{corollaryalpha}[theoremalpha]{Corollary}

\theoremstyle{remark}
\newtheorem{remark}[lemma]{Remark}
\newtheorem{example}[lemma]{Example}

\newcommand{\HF}{\operatorname{HF}}

\newcommand{\Span}{\operatorname{span}}
\newcommand{\img}{\operatorname{img}}

\newcommand{\Sym}{\operatorname{Sym}} %done

\newcommand{\supp}{\operatorname{supp}}

\newcommand{\Cox}{\operatorname{Cox}}

\newcommand{\remd}{\operatorname{remd}}

\newcommand{\indeg}{\operatorname{index.deg}}
\newcommand{\moddeg}{\operatorname{mod.deg}}

\newcommand{\doot}{\bullet}

\newcommand{\Alt}{\bigwedge\nolimits}
\renewcommand{\aa}{\mathbf a}
\newcommand{\bb}{\mathbf b}

\newcommand{\dd}{\mathbf d}

\newcommand{\vv}{\mathbf v}
\newcommand{\ww}{\mathbf w}
\newcommand{\xx}{\mathbf x}
\newcommand{\yy}{\mathbf y}
\newcommand{\rr}{\mathbf r}

\newcommand{\nn}{\mathbf n}

\newcommand{\fF}{\mathbf F}

\newcommand{\one}{\mathbf 1}
\newcommand{\zero}{\mathbf 0}

\renewcommand{\H}{\operatorname{H}}

\newcommand{\cI}{\mathcal{I}}

\renewcommand{\O}{\mathcal{O}}

\newcommand{\K}{\mathbb{K}}
\renewcommand{\L}{\mathbb{L}}

\newcommand{\N}{\mathbb{N}}
\renewcommand{\P}{\mathbb{P}}

\newcommand{\R}{\mathbb{R}}

\newcommand{\Z}{\mathbb{Z}}

\newcommand{\n}{\mathfrak{n}}

\renewcommand{\R}{\mathfrak{R}}
\renewcommand{\L}{\mathfrak{L}}

\title{Asymptotic Syzygies in the Setting of Semi-Ample Growth}

\author{Juliette Bruce}
\address{Department of Mathematics, University of Wisconsin, Madison, WI}
\email{\href{mailto:juliette.bruce@math.wisc.edu}{juliette.bruce@math.wisc.edu}}
\urladdr{\url{http://math.wisc.edu/~juliettebruce/}}

\thanks{The author was partially supported by the NSF GRFP under Grant No. DGE-1256259.}


\begin{document}

\begin{abstract}
We study the asymptotic non-vanishing of syzygies for products of projective spaces. Generalizing the monomial methods of Ein, Erman, and Lazarsfeld \cite{einErmanLazarsfeld16} we give an explicit range in which the graded Betti numbers of $\P^{n_1}\times \P^{n_2}$ embedded by $\O_{\P^{n_1}\times\P^{n_2}}(d_1,d_2)$ are non-zero. These bounds provide the first example of how the asymptotic syzygies of a smooth projective variety whose embedding line bundle grows in a semi-ample fashion behave in nuanced and previously unseen ways. 
 \end{abstract}

\maketitle

%\tableofcontents

\setcounter{section}{1}

The goal of this paper is to initiate the study of the asymptotic behavior of the syzygies of a smooth projective variety as the embedding line bundle grows in a semi-ample fashion. We show that for the prototypical example of such varieties, the product of two projective spaces, the asymptotic behavior is more complicated and nuanced than in the case when the positivity grows in an ample fashion. In particular, we show that the non-vanishing theorems of Ein and Lazarsfeld and others \cites{conca18,einLazarsfeld12,einErmanLazarsfeld16,ermanYang18} do not describe the non-vanishing syzygies of products of projective space in the setting of semi-ample asymptotics.

More specifically, fix $\nn=(n_1,n_2)\in \Z_{\geq1}^2$ and set $\P^{\nn}:=\P^{n_1}\times\P^{n_2}$. Given $\bb=(b_1,b_2)\in \Z^2$, we let
\[
\O_{\P^{\nn}}(\bb)\coloneqq \pi_{1}^{*}\O_{\P^{n_1}}(b_1)\otimes \pi_{2}^{*}\O_{\P^{n_2}}(b_2) \,,
\]
where $\pi_i$ is the projection from $\P^{\nn}$ to $\P^{n_i}$. If $\dd\in \Z_{\geq1}^2$ then $\O_{\P^{\nn}}(\dd)$ is very ample, and so defines an embedding:
\[
\begin{tikzcd}[column sep = 3 em, row sep = 4em]
\P^{\nn}=\P^{n_1}\times\P^{n_2} \arrow[rr,"\iota_{\dd}"]& &\P \H^0\left(\P^{\nn}, \O_{\P^{\nn}}(\dd)\right)\cong\P^{r_{\nn,\dd}}
\end{tikzcd}.
\]
We call this the \textbf{$\dd$-uple Segre-Veronese map}. We are interested in studying the asymptotic behavior of the syzygies of $\P^{\nn}$ under the $(d_1,d_2)$-uple Segre-Veronese embedding as $d_1$ or $d_2$ goes to infinity. More generally, following the work of Green \cites{green84-I,green84-II}, we also study the syzygies of other line bundles on $\P^{\nn}$, as this often provides a more unified perspective, see for example \cite[Theorem~2.2]{green84-II}, ~\cite[Theorem~2]{einLazarsfeld93}, and \cite[Theorem~4.1]{einLazarsfeld12}. Thus, let
\[
S(\bb;\dd)=\bigoplus_{k\in \Z}H^0(\P^{r_{\nn,\dd}}, (\iota_{\dd})_{*}(\O_{\P^{\nn}}(\bb)(k)))
\]
be the graded section ring of the pushforward of $\O_{\P^{\nn}}(\bb)$ along the map $\iota_{\dd}$. We consider $S(\bb;\dd)$ as an $R$-module where $R=\Sym H^0\left(\P^{\nn}, \O_{\P^{\nn}}(\dd)\right)$ is the homogeneous coordinate ring of $\P^{r_{\nn,\dd}}$. Note if $\dd\gg0$ then the line bundle $\O_{\P^n}(\dd)$ will be normally generated, in which case $M(\zero;\dd)$ is isomorphic to the homogeneous coordinate ring of $\P^{\nn}$ as a subvariety of $\P^{r_{\nn,\dd}}$.  

\begin{remark}
Using the description of the $\dd$-uple Segre-Veronese map given above, one sees that
\[
r_{\nn,\dd}\coloneqq\binom{d_1+n_1}{n_1}\binom{d_2+n_2}{n_2}-1\in O\left(d_1^{n_1}d_2^{n_2}\right).
\]
Throughout the paper, we will use big-O notation for multivariate functions as follows: if $f$ and $g$ are $\mathbb{R}$-valued functions defined on some domain $U\subset \mathbb{R}^{n}$, then we write $f(\xx)\in O(g(\xx))$ as $\xx\to\infty$ if and only if there exists constants $C>0$ and $M>0$ such that $|f(\xx)|\leq C|g(\xx)|$ for all $\xx\in U$, with $\|\xx\|_{\infty}\geq M$. 
\end{remark}

By studying syzygies of $\P^{\nn}$, we mean studying the minimal graded free resolution of $S(\bb;\dd)$ as an $R$-module. The Hilbert Syzygy Theorem \cite[Theorem~1.1]{eisenbud05} implies that the minimal graded free resolution of $S(\bb;\dd)$ over $R$ is of the form
\[
\begin{tikzcd}[column sep = 3em]
0 & \lar S(\bb;\dd) & \lar F_{0} & \lar F_{1} & \lar \cdots & \cdots & \lar F_{r_{\nn,\dd}} & \lar 0\,,
\end{tikzcd}
\]
where $F_{p}$ is a finitely generated graded free $R$-module. If we let
\[
K_{p,q}\left(\nn,\bb;\dd\right):=\Span_{\K}~\left\langle
\begin{matrix}
~\text{minimal generators of $F_p$} ~\\
\text{of degree $(p+q)$}
\end{matrix}
 \right\rangle
\]
be the finite dimensional vector space of minimal syzygies of homological degree $p$ and degree $(p+q)$, then $F_{p}$ is isomorphic to $\bigoplus_{q}K_{p,q}\left(\nn,\bb;\dd\right)\otimes_{\K}R\left(-(p+q)\right)$. When $\bb=\zero$, we often write $K_{p,q}(\nn;\dd)$ for $K_{p,q}(\nn,\zero;\dd)$.  The main question we are interested in is the following.

\begin{question}\label{quest:main}
If $\dd\gg0$, then for what $p$ and $q$ is $K_{p,q}(\nn,\bb;\dd)\neq0$?
\end{question}

Considerations of Castelnuovo-Mumford regularity imply that if $\dd\gg0$ then $K_{p,q}(\nn,\bb;\dd)$ for all $q>|\nn|=n_1+n_2$. Thus, Question~\ref{quest:main} is of primary interest when $1\leq q\leq|\nn|$. (See Proposition~\ref{prop:regularity} for a precise description of how large $\dd$ must be, relative to $\bb$ and $\nn$, for $K_{p,q}(\nn,\bb;\dd)$ for all $q>|\nn|=n_1+n_2$.)

As a running example, consider $\P^1\times\P^1$ when $\bb=(0,0)$ and $\dd=(d_1,d_2)$. If both $d_1,d_2\to\infty$, then the work of Ein and Lazarsfeld provides an answer to Question~\ref{quest:main} \cite{einLazarsfeld12}. For example, Ein and Lazarsfeld's work implies that if $d_1,d_2\to\infty$ then $K_{p,2}(\nn,\zero;\dd)\neq0$ for $100\%$ of possible $p$'s  \cite[Theorem A]{einLazarsfeld12}. That said, Ein and Lazarsfeld's results require the embedding line bundle to grow in an ample fashion, i.e. the embedding bundle needs to be of the form $B+dA$ where $A$ is ample. Thus, if we fix either $d_1$ or $d_2$, then Ein and Lazarsfeld's non-vanishing results no longer apply. 

For example, if we fix $d_2$ and allow only $d_1\to\infty$, then the embedding line bundle $\O_{\P^{\nn}}(d_1,d_2)$ grows like $\O_{\P^{\nn}}(0,1)$, which is semi-ample. Recall a line bundle $L$ on a smooth variety is \textbf{semi-ample} if the complete linear series $|kL|$ is base point free for some $k\geq1$. The prototypical example of semi-ample line bundles are $\O_{\P^{\nn}}(1,0)$ and $\O_{\P^{\nn}}(0,1)$ on $\P^{\nn}$.

The difference between the cases of ample and semi-ample growth can be visualized if we view the sequence of embedding line bundles as a sequence of points inside the nef cone of $\P^1\times \P^1$.  
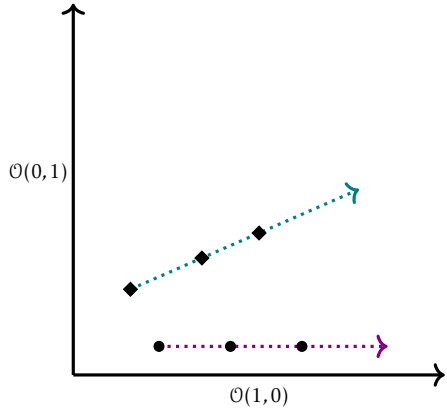
\begin{center}
\begin{figure}[H]
\begin{tikzpicture}[xscale=0.38, yscale=0.38]
\draw[->, very thick] (3,0) -- (16,0);
\draw[->, very thick] (3,0) -- (3,13);
\draw (2,0) node{};
\draw (3.7,9,5) node{\footnotesize $\mathcal{O}(0,1)$};
\draw (3,-.5) node{};
\draw (9.5,-.75) node{\footnotesize $\mathcal{O}(1,0)$};
\draw (17,-.5) node{};
\draw[->,very thick,dotted,teal] (5,3)--(13,6.5);
\node [diamond, draw, fill=black,scale=0.4] (a2) at (5,3) {};
\node [diamond, draw, fill=black,scale=0.4] (a2) at (7.5,4.09375) {};
\node [diamond, draw, fill=black,scale=0.4] (a2) at (9.5,4.96875) {};

\draw[->,very thick,dotted,violet] (6,1)--(14,1 );
\node [circle, draw, fill=black,scale=0.4] (a2) at (6,1) {};
\node [circle, draw, fill=black,scale=0.4] (a2) at (8.5,1) {};
\node [circle, draw, fill=black,scale=0.4] (a2) at (11,1) {};
\end{tikzpicture}
\end{figure}
\end{center}
The case covered by Ein and Lazarsfeld corresponds to the sequence of points going to infinity along a line of positive slope, for example, the sequence of points (diamonds) along the teal line. The case of semi-ample growth not covered by Ein and Lazarsfeld's results corresponds to the sequence of points going to infinity along a ray parallel to one of the axes. For example, see the points (circles) along the purple line. 

Interestingly, the syzygies of $\P^1\times\P^1$ in the semi-ample case behave differently than in the ample case. For instance, it is no longer true that in the limit $K_{p,2}((1,1);\dd)\neq0$ for $100\%$ of possible $p$'s. More precisely, following the notation of \cite{ermanYang18}, set
\[
\rho_q(\nn;\dd)\coloneqq\frac{\#\left\{p\in\N |\; \big| \; K_{p,q}(\nn;\dd)\neq0\right\}}{r_{\dd}}\,,
\]
which by the Hilbert Syzygy Theorem is the percentage of degrees in which non-zero syzygies appear \cite[Theorem~1.1]{eisenbud05}. A result of Lemmens's \cite{lemmens18} implies that:
\begin{align*}
\lim_{d_1\to\infty}\rho_2((1,1);\dd)=\lim_{d_1\to\infty}\quad\begin{matrix}
\text{\% of $p$ where}\\
K_{p,2}((1,1);\dd)\neq0
\end{matrix} \quad&=1-\frac{2}{d_2+1}\,.
\end{align*}
Thus, syzygies in the setting of semi-ample growth can behave differently than is suggested by the work of Ein and Lazarsfeld \cite{einLazarsfeld12,einErmanLazarsfeld16}. Further, the fact that this limit is not zero shows that syzygies in the case of semi-ample growth also do not behave as suggested by Green's work on syzygies of curves \cite{green84-I,green84-II}. Hence the asymptotic behavior of syzygies under semi-ample growth is not controlled by either the dimension or the Iitaka dimension of the embedding line bundle.

Our main result is the following, which, given $q$, gives a range for $p$ for which these vector spaces of syzygies, $K_{p,q}(\nn;\dd)$, are non-zero. 

\begin{theoremalpha}\label{thm:main}
Fix $\nn=(n_1,n_2)\in \Z_{\geq1}^{2}$, $\dd=(d_1,d_2)\in \Z_{>1}^2$, and an index $1\leq q\leq |\nn|$. If $d_1>q$ and $d_2>q$ then $K_{p,q}(\nn;\dd)\neq0$ for all $p$ in the range:
\[
\min\left\{\binom{d_1+i}{i}\binom{d_2+j}{j}\right\}_{\substack{i+j=q \\ 0\leq i \leq n_1 \\ 0\leq j \leq n_2}}-(q+2) \leq p \leq r_{\nn,\dd}-\min\left\{\binom{d_1+n_1-i}{n_1-i}\binom{d_2+n_2-j}{n_2-j}\right\}_{\substack{i+j=q \\ 0\leq i \leq n_1 \\ 0\leq j \leq n_2}}-(|\nn|+1).
\]
\end{theoremalpha}

Notice that these bounds depend on both $d_1$ and $d_2$. In particular, that asymptotic behavior is dependent, in a nuanced way, on the relationship between $d_1$ and $d_2$. Again, this underscores the complicated asymptotic behaviors possible for syzygies under semi-ample growth. 

In order to highlight this behavior, and explain the terms appearing in the bounds of Theorem~\ref{thm:main}, let us consider what can occur when $q=2$. In this case, assuming $n_1,n_2\geq2$, the main terms of the bounds in Theorem~\ref{thm:main} can be written as 
\begin{equation*}\label{intro:example}
\min\left\{
\frac{d_1^2}{2}, \;
d_1d_1,  \;
\frac{d_2^2}{2}
\right\}-O\left(\begin{matrix}\text{lower ord.}\\ \text{terms}\end{matrix}\right)
\leq p\leq r_{\nn,\dd}-\min\left\{ 
\frac{d_1^{n_1}d_2^{n_2-2}}{n_1!(n_2-2)!}, \;
\frac{d_1^{n_1-1}d_2^{n_2-1}}{(n_1-1)!(n_2-1)!} , \;
\frac{d_1^{n_1-2}d_2^{n_2}}{(n_1-2)!n_2!} 
\right\}
-O\left(\begin{matrix}\text{lower ord.}\\ \text{terms}\end{matrix}\right).
\end{equation*}
Focusing our attention on the upper bounds, we see that there are roughly three cases. If $d_1\gg d_2$, we expect the upper bound to be approximately $r_{\nn,\dd}-Cd_1^{n_1-2}d_2^{n_2}$ where $C$ is a constant. On the other hand, if $d_1\sim d_2$, then the upper bound is roughly $r_{\nn,\dd}-C'd_1^{n_1-1}d_2^{n_2-1}$ for some constant $C'$. Finally, if $d_2\gg d_1$, we expect the upper bound to be approximately $r_{\nn,\dd}-C''d_1^{n_1}d_2^{n_2-2}$ for some constant $C''$. 

For larger $q$, the number of cases, and the distinctions between them, become much more complicated. We propose the following rough heuristic for thinking about the bounds appearing in Theorem~\ref{thm:main}. The lower bounds reflect the asymptotic syzygies of restricting $\O_{\P^{\nn}}(\dd)$ to $\P^{i}\times\P^{q-i}\subset \P^{\nn}$ as $i$ varies. Similarly, the upper bounds reflect asymptotic syzygies of restricting $\O_{\P^{\nn}}(\dd)$ to $\P^{n_1-i}\times\P^{n_2-j}\subset \P^{\nn}$ for $i+j=q$. 

In fact, when proving Theorem~\ref{thm:main} we explicitly construct non-trivial syzygies in the given ranges, and in a sense, these syzygies naturally live on subvarieties of the form $\P^{i}\times\P^{j}\subset \P^{\nn}$ where $i+j=q$. This can be seen in a technical way in that we deduce Theorem~\ref{thm:main} from Theorem~\ref{thm:special} via a lifting argument. 

As an immediate corollary of Theorem~\ref{thm:main}, and a generalization of the example of $\P^1\times\P^1$ discussed above, we are able to provide a lower bound on the percentage of degrees in which non-zero syzygies asymptotically appear. In the following corollary, we let $C_{i,j}=\frac{n_1!n_2!}{(n_1-i)!(n_2-j)!}$ and $D_{i,j}=\frac{n_1!n_2!}{i!j!}$.

\begin{corollaryalpha}\label{cor:main}
Fix $\nn=(n_1,n_2)\in \Z_{\geq1}^{2}$, $\dd=(d_1,d_2)\in \Z_{>1}^2$, and an index $1\leq q\leq |\nn|$. If $d_1>q$ and $d_2>q$ then
\begin{align*}
\rho_q(\nn;\dd)\geq1-\sum_{\substack{i+j=q \\ 0\leq i \leq n_1 \\ 0\leq j \leq n_2}}\left(
\frac{C_{i,j}}{d_1^id_2^j}+\frac{D_{i,j}}{d_1^{n_1-i}d_2^{n_2-j}}+O\left(\frac{d_1+d_2}{d_1^{i+1}d_2^{j+1}}+\frac{d_1+d_2}{d_1^{n_1-i+1}d_2^{n_2-j+1}}\right)
\right)\,.
\end{align*}
\end{corollaryalpha}

\begin{example}
If we let $\nn=(1,5)$ and $q=2$, then by Corollary~\ref{cor:main} we see that
\[
\rho_2((1,5);\dd)\geq1-\frac{20}{d_2^2}-\frac{60}{d_1d_2^3}-\frac{5}{d_1d_2}-\frac{120}{d_2^4}-O\left(\begin{matrix}\text{lower ord.}\\ \text{terms}\end{matrix}\right)\,.
\]
In particular, if $d_2$ is fixed and $d_1\to\infty$, then the limit of $\rho_q(\nn;\dd)$ is greater than or equal to $1-\frac{20}{d^2_2}-\frac{120}{d_2^4}$.
\end{example}

In the setting of ample growth, this recovers the results of Ein and Lazarsfeld: namely, if both $d_1\to\infty$ and $d_2\to\infty$, then $\rho_q(\nn;\dd)\to1$. At the other extreme, if $d_2$ is fixed and only $d_1\to\infty$, then 
\[
\lim_{d_1\to\infty}\rho_{q}(\nn;\dd)\geq1-\frac{n_{2}!}{(n_{2}-q)!d_2^q}-\frac{n_2!}{(n_1-q)!d_2^{n_1+n_2-q}}.
\]
In particular, in this case, we do not believe $\rho_q(\nn;\dd)$ will approach 1. Proving this would require a vanishing result for asymptotic syzygies, which is open even in the ample case. See \cite[Conjecture~7.1, Conjecture~7.5]{einLazarsfeld12}.

Under mild hypotheses, we are able to generalize Theorem~\ref{thm:main} to describe the asymptotic non-vanishing of syzygies for other line bundles on $\P^{\nn}$.

\begin{theoremalpha}\label{thm:main2}
Fix $\nn=(n_1,n_2)\in \Z_{\geq1}^{2}$, $\dd=(d_1,d_2)\in \Z_{>1}^2$, $\bb\in \Z_{\geq0}^2$ and an index $1\leq q\leq |\nn|$. If $d_1>q+b_1$, $d_2>q+b_2$,
\begin{equation}\label{eq:thm2}
\frac{d_1}{d_2}b_2-b_1<n_1+1,\quad \quad \text{and} \quad \quad \frac{d_2}{d_1}b_1-b_2<n_2+1\,,
\end{equation}
then $K_{p,q}(\nn,\bb;\dd)\neq0$ for all $p$ in the range:
\[
\min\left\{\binom{d_1+i}{i}\binom{d_2+j}{j}\right\}_{\substack{i+j=q \\ 0\leq i \leq n_1 \\ 0\leq j \leq n_2}}-(q+2) \leq p \leq r_{\nn,\dd}-\min\left\{\binom{d_1+n_1-i}{n_1-i}\binom{d_2+n_2-j}{n_2-j}\right\}_{\substack{i+j=q \\ 0\leq i \leq n_1 \\ 0\leq j \leq n_2}}-(|\nn|+1)\,.
\]
\end{theoremalpha}

In many ways, this theorem mimics Theorem~\ref{thm:main}. For example, since $\bb$ is fixed, the bounds governing non-vanishing depend only on $\dd$ and are the same as the bounds in Theorem~\ref{thm:main}. The main difference between these theorems is that when $\bb\neq\zero$ the section module $S(\bb;\dd)$ need not be Cohen-Macaulay as an $R$-module. Our methods require, in a crucial way, that $S(\bb;\dd)$ be Cohen-Macaulay. The conditions in \eqref{eq:thm2} exactly characterize when this occurs. 

Our proof of Theorem~\ref{thm:main} is based on generalizing the monomial methods of \cite{einErmanLazarsfeld16} to explicitly construct non-zero syzygies in the given ranges after having quotient by a regular sequence. The general idea is that given a linear regular sequence on $S(\bb;\dd)$ and an element $f\in R$, not contained in the ideal generated by the regular sequence, it is possible to construct non-zero syzygies in a range determined by the regular sequence and the element $f$. More specifically, if $I\subset R$ is an ideal generated by linear forms that are a regular sequence on $S(\bb;\dd)$, satisfying a few technical conditions, and $f\in R\setminus I$ is a monomial of degree $q$, then there exists a subset $L(f)\subset (I:_{R}f)$ such that $K_{p,q}(\nn,\bb;\dd)\neq0$ for all $p$ in the range:
\[
\#L(f) \leq p \leq \dim_{\K} (I:_R f)_{1}- \dim_{\K} I_{1}\,.
\] 
In \cite{einErmanLazarsfeld16} Ein, Erman, and Lazarsfeld prove a similar result for a single projective space. In this case, they work with a particular linear monomial regular sequence and define $L(f)$ in terms of monomials dividing $f$. However, these methods cannot be directly applied to a product of projective spaces. 

First, the case of a product of projective spaces is substantially complicated by the fact that there are no monomial regular sequences of length $|\nn|+1$ on either the $\Z^2$-graded Cox ring of $\P^{\nn}$, denoted $\Cox(\P^{\nn})$ (see \cite{cox95}), or the $\Z$-graded homogeneous coordinate ring of $\P^{\nn}$ embedded by $\O_{\P^{\nn}}(\dd)$. Instead, we work with a set of bi-degree $\dd=(d_1,d_2)$ elements of the Cox ring of $\P^{\nn}$, which, while not a regular sequence of length $|\nn|+1$ on $\Cox(\P^{\nn})$, corresponds to a regular sequence of length $|\nn|+1$ on the homogeneous coordinate ring of $\P^{\nn}$ embedded by $\O_{\P^{\nn}}(\dd)$. For example, if $\nn=(2,4)$ and $\dd=(d_1,d_2)$ then the sequence we use is:
\begin{align*}
g_{0}&=x_0^{d_1}y_0^{d_2}\\
g_{1}&=x_1^{d_1}y_0^{d_2}+x_0^{d_1}y_1^{d_2} \\
g_{2}&=x_{2}^{d_1}y_{0}^{d_2}+x_1^{d_1}y_1^{d_2}+x_0^{d_1}y_2^{d_2} \\
g_{3}&=x_2^{d_1}y_1^{d_2}+x_1^{d_1}y_2^{d_2}+x_0^{d_1}y_3^{d_2} \\
g_{4}&=x_{2}^{d_1}y_{2}^{d_2}+x_1^{d_1}y_{3}^{d_2}+x_0^{d_1}y_{4}^{d_2} \\
g_{5}&=x_{2}^{d_1}y_{3}^{d_2}+x_1^{d_1}y_{4}^{d_2}\\
g_{6}&=x_2^{d_1}y_{4}^{d_2}.
\end{align*}
Put differently, we work with a set of bi-degree $\dd=(d_1,d_2)$ elements of the Cox ring of $\P^{\nn}$ that is not a regular sequence on $\Cox(\P^{\nn})$, but which define an ideal in $\Cox(\P^{\nn})$ supported on the irrelevant ideal of $\P^{\nn}$. Thus, in the language of \cite{berkesch17}, we work with a virtual regular sequence of length $|\nn|+1$ of bi-degree $\dd=(d_1,d_2)$ on $\Cox(\P^{\nn})$. These forms, when considered as degree one elements of the homogeneous coordinate ring of $\P^{\nn}$ embedded by $\O_{\P^{\nn}}(\dd)$, are a regular sequence of length $|\nn|+1$. 

\begin{example}
Continuing the example when $\nn=(2,4)$ from above, the elements $g_{0},g_{1},\ldots,g_{6}$ do not form a regular sequence on $\Cox(\P^{2}\times\P^{4})\cong \K[x_0,x_1,x_2,y_0,y_1,y_2,y_3,y_{4}]$. In particular, $\langle g_{0},g_{1},\ldots,g_{6}\rangle$ has $\langle x_0,x_1,x_2\rangle$ as one of its associated primes, so the $g_{i}$'s do not even form a system of parameters on $\Cox(\P^2\times\P^4)$. That said, one can show that $g_{0},g_{1},\ldots,g_{6}$ is supported on $\langle x_0,x_1,x_2\rangle\cap \langle y_{0},y_{1},y_{2},y_{3},y_{4}\rangle$. 
\end{example}

Working with such a regular sequence poses significant new challenges. For example, in \cite{einErmanLazarsfeld16} the authors work with a monomial regular sequence, and so computing the analog of $(I:_Rf)$ is amenable to monomial techniques. The regular sequence we work with, on the other hand, is complicated, and computing $(I:_{R}f)$ is in general difficult. 

In fact, we devote all of Section~\ref{sec:ideal-membership} to developing methods for understanding $(I:_{R}f)$. The central theme is to exploit the fact that our regular sequence, while not monomial, has a large number of symmetries. That is, the ideal generated by the regular sequence is homogeneous with respect to a number of non-trivial non-standard gradings. These gradings, when combined with a series of spectral sequence arguments, eventually allow us to describe $(I:_{R}f)$.

A second subtle challenge is defining the correct subset of $(I:_{R}f)$, from which to construct non-trivial syzygies. In particular, since $I$ is not generated by monomials, the notion of one monomial dividing another in $R/I$ is quite nuanced. This means the definition of $L(f)$ used in \cite{einErmanLazarsfeld16} for a single projective space does not generalize to the case of a product of projective spaces. Instead, we make use of a non-standard grading for which $I$ is homogeneous, and define $L(f)$ in terms of certain degrees in this special grading. 

Finally, we note that Theorem~\ref{thm:main} is not sharp. One source of error is that we are unable to fully describe $(I:_{R}f)$, and a better understanding of this ideal would result in sharper non-vanishing statements. That said, we do believe that the upper bounds in Theorem~\ref{thm:main} are asymptotically sharp.  

%%%%%%%%%%%%%%%%%%%%%%%%%%%%%%%%%%%%%%
%%%%%%%%%%%%%%%%%%%%%%%%%%%%%%%%%%%%%%
%%%%%%%%%%%%%%%%%%%%%%%%%%%%%%%%%%%%%%
%%%%%%%%%%%%%%%%%%%%%%%%%%%%%%%%%%%%%%

The paper is organized as follows: \S~\ref{sec:set-up} gathers background results and sets up the problem. \S~\ref{sec:reg-seq} introduces the regular sequence crucial to our methods, and in \S~\ref{sec:ideal-membership} we study the ideal membership question for the ideal generated by this regular sequence. In \S~\ref{sec:monomial-techniques} we develop the monomial methods we use to construct non-trivial syzygies, and \S~\ref{sec:fqk} presents the exact monomials we will use. \S~\ref{sec:special-case} contains the key case of $K_{p,q}(\nn,\bb;\dd)$ for $\P^{q-k}\times\P^{k}$. Finally, \S~\ref{sec:proof-of-theorem} contains the proof of Theorem~\ref{thm:main}.

%%%%%%%%%%%%%%%%%%%%%%%%%%%%%%%%%%%%%%
%%%%%%%%%%%%%%%%%%%%%%%%%%%%%%%%%%%%%%
%%%%%%%%%%%%%%%%%%%%%%%%%%%%%%%%%%%%%%
%%%%%%%%%%%%%%%%%%%%%%%%%%%%%%%%%%%%%%

\section*{Acknowledgements}  I would like to thank Daniel Erman, Mois\'es Herrad\'on Cueto, Kit Newton, Solly Parenti, Claudiu Raicu, and Melanie Matchett Wood for their helpful conversations and comments. The computer algebra system \texttt{Macaulay2} \cite{M2} provided valuable assistance throughout this work.

%%%%%%%%%%%%%%%%%%%%%%%%%%%%%%%%%%%%%%
%%%%%%%%%%%%%%%%%%%%%%%%%%%%%%%%%%%%%%
%%%%%%%%%%%%%%%%%%%%%%%%%%%%%%%%%%%%%%
%%%%%%%%%%%%%%%%%%%%%%%%%%%%%%%%%%%%%%

\section*{Notation}  Throughout we work over a field $\K$. When clear we generally admit the reference to the field, and so write $\P^{r}$ for  $\P^{r}_{\K}\coloneqq \P(\K^{r+1})$. When referring to vectors (i.e. elements of $\Z^{n}$ or $\K^n$) we normally use a bold font for example $\aa,\bb,\dd,\vv,\ww$. Given a vector $\vv=(v_1,v_2,\ldots,v_{n})$ we denote the sum $v_1+v_2+\cdots+v_n$ by $|\vv|$. For the sake of notational hygiene we abuse notation slightly and write $\Z^2_{\geq1}$ for $\left(\Z_{\geq1}\right)^2$, that is tuples $(a,b)\in \Z^2$ such that $a\geq 1$ and $b\geq1$. Likewise for $\Z_{\geq0}^2$. For brevity we write $\one$ for $(1,1)\in\Z^2$ and $\zero$ for $(0,0)\in\Z^2$.

%%%%%%%%%%%%%%%%%%%%%%%%%%%%%%%%%%%%%%%%%%%%%%%%%%%%%%%%
%%%%%%%%%%%%%%%%%%%%%%%%%%%%%%%%%%%%%%%%%%%%%%%%%%%%%%%%
\section{Background and Set Up}\label{sec:set-up}
%%%%%%%%%%%%%%%%%%%%%%%%%%%%%%%%%%%%%%%%%%%%%%%%%%%%%%%%%%%%%%%%%%%%%%%%%%%%%%%%%%%%%%%%%%%%%%%%%%%%%%%%%%%%%%%%

Fixing $\nn=(n_1,n_2)\in \Z^2_{\geq1}$, we let $S'=\K[x_{0},x_{1},\ldots,x_{n_1}]$ and $S''=\K[y_{0},y_{1},\ldots,y_{n_2}]$ be standard $\Z$-graded polynomial rings, and set $S=S'\otimes_{\K}S''$ with the induced $\Z^{2}$-multigrading. Concretely $S$ is isomorphic to the bi-graded polynomial ring $\K[x_0,x_1,\ldots,x_{n_1},y_{0},y_{1}\ldots,y_{n_2}]$ where $\deg(x_i)=(1,0)\in \Z^2$ and $\deg(y_j)=(0,1)\in \Z^2$ for every $i=0,1,\ldots,n_1$ and $j=0,1,\ldots,n_2$. Moreover, $S$ is isomorphic to Cox ring of $\P^{\nn}$, which we generally denote $\Cox(\P^{\nn})$ (see \cite{cox95}). Since $S$ is $\Z^2$-graded there is a natural decomposition of $\K$-vector spaces
\[
S\cong \bigoplus_{\aa\in\Z^2} S_{\aa},
\]
where $S_{\aa}$ is the $\K$-vector space spanned by monomials in $S$ of bi-degree $\aa$. The Hilbert function of $S$ is the function $\HF(\aa,S)=\dim_{\K} S_{\aa}$. Similarly given an ideal $J\subset S$ that is homogeneous with respect to this $\Z^2$-grading, we write $J_{\aa}$ for the $\K$-vector space spanned by monomials in $J$ of bi-degree $\aa$, and the Hilbert function of $J$ is the function $\HF(\aa,J)=\dim_{\K} J_{\aa}$.

\begin{defn}
Given $\bb\in \Z^2$ and $\dd\in \Z^2_{\geq1}$ we define the \textbf{bi-graded Veronese module} of $S$ to be 
\[
S(\bb;\dd)\coloneqq\bigoplus_{k\in \Z}S_{k\dd+\bb}\subset S,
\]
which we consider as a $\Z$-graded $R=\Sym S_{\dd}$ module. 
\end{defn}

More specifically a generator $\ell$ of $R$ corresponds to a monomial $m\in S_{\dd}$, and then $\ell$ acts on $S(\bb;\dd)$ via multiplication by this monomial $m$. Further, the degree $k$ piece of $S(\bb;\dd)$ is $S_{k\dd+\bb}$, and so the degree one piece is $S_{\dd+\bb}$. Now as an $R$-module $S(\bb;\dd)$ is isomorphic to the $\Z$-graded homogeneous coordinate ring of $\P^{\nn}$ embedded by $\O_{\P^{\nn}}(\dd)$. 

Given $p,q\in \N$ we define \textbf{$(p,q)$-th Koszul cohomology group of} $S(\bb;\dd)$ to be the cohomology of the sequence:
\begin{equation}\label{eq:koszul-R}
\begin{tikzcd}[column sep = 3em]
\cdots \rar{}& \Alt^{p+1}R_{1}\otimes S_{(q-1)\dd+\bb}\rar{\partial_{p+1,q-1}}&\Alt^{p}R_{1}\otimes S_{q\dd+\bb}\rar{\partial_{p,q}}&\Alt^{p-1}R_{1}\otimes S_{(q+1)\dd+\bb}\rar{}&\cdots
\end{tikzcd}
\end{equation}
where the differentials are given 
\begin{align*}
\partial_{p+1,q-1}\left(m_0\wedge m_1\wedge\cdots\wedge m_p\otimes f\right)&=\sum_{i=0}^{p}(-1)^i m_0\wedge m_1\wedge\cdots\wedge \hat{m}_i\wedge \cdots\wedge m_p\otimes m_if \\
\partial_{p,q}\left(m_1\wedge m_2\wedge\cdots\wedge m_p\otimes f\right)&=\sum_{i=1}^{p}(-1)^i m_1\wedge m_2\wedge\cdots\wedge \hat{m}_i\wedge \cdots\wedge m_p\otimes m_if.
\end{align*}
As in the introduction, we denote this by $K_{p,q}(\nn,\bb;\dd)$, and note that $K_{p,q}(\nn,\bb;\dd)\cong K_{p,q}\left(\P^{\nn},\O_{\P^{\nn}}(\bb);\O_{\P^{\n}}(\dd)\right)$.

That said it will be helpful for us to realize the Koszul complex in \eqref{eq:koszul-R} above in a different way. Towards this, notice that there exist maps:
\begin{center}
\begin{tikzcd}[column sep = 3.5em]
R\rar[two heads]{} & S(\zero,\dd) \rar[hook] & S
\end{tikzcd}.
\end{center}
Moreover, when restricted to the degree one piece of $R$, and the subsequent images, these maps give natural isomorphisms
\begin{center}
\begin{tikzcd}[column sep = 3.5em]
R_{1}\arrow[r, leftrightarrow, "\sim"] & S(\zero;\dd)_{1} \arrow[r, leftrightarrow, "\sim"] & S_{\dd}.
\end{tikzcd}
\end{center}
Thus, the $R$-module action of $R_{1}$ on $S$ is the same as the $S$-module action of $S_{\dd}$ on $S$, and so the Koszul complex in \eqref{eq:koszul-R} is naturally isomorphic to the following Koszul complex:
\begin{equation}\label{eq:kozul-S}
\begin{tikzcd}[column sep = 3em]
\cdots \rar{}& \Alt^{p+1}S_{\dd}\otimes S_{(q-1)\dd+\bb}\rar{\partial_{p+1,q-1}}&\Alt^{p}S_{\dd}\otimes S_{q\dd+\bb}\rar{\partial_{p,q}}&\Alt^{p-1}S_{\dd}\otimes S_{(q+1)\dd+\bb}\rar{}&\cdots
\end{tikzcd}
\end{equation}
where the differentials are defined in an analogous way. So the cohomology of \eqref{eq:kozul-S} is isomorphic to $K_{p,q}(\nn,\bb;\dd)$.

We end this section by noting that considerations of Castelnuovo-Mumford regularity show that if $\dd\gg0$, relative to $\bb$, then $K_{p,q}(\nn,\bb;\dd)=0$ for $q>|\nn|$. In particular, if $\bb=\zero$ then $K_{p,q}(\nn,\zero;\dd)=0$ for $q>|\nn|$ for all choices of $\dd\in \Z^2_{\geq1}$. 

\begin{example}
If $\bb\neq\zero$ then it is not the case that $K_{p,|\nn|}(\nn,\zero;\dd)=0$ for all choices of $\dd\in \Z^2_{\geq1}$. For example, using arguments similar to those in Proposition~\ref{prop:regularity} one can show that if $\nn=(1,3)$, $\dd=(3,3)$, and $\bb=(-2,-1)$ then there exists $p$ such that $K_{p,|\nn|}(\nn,\bb;\dd)\neq0$.\end{example}

\begin{prop}\label{prop:regularity}
Fix $\nn=(n_1,n_2)\in \Z^2_{\geq1}$, $\dd=(d_1,d_2)\in \Z^2_{\geq1}$, and $\bb\in \Z^2$. If the following two pairs of inequalities are satisfied then $K_{p,q}(\bb,\bb;\dd)=0$ for $q>|\nn|$:
\begin{align}
d_1+b_1n_2>-n_1-1 \quad \quad &\text{or} \quad \quad d_2+b_2n_2<0, \label{eq:reg-1} \\
d_1+b_1n_1<0 \quad \quad &\text{or} \quad \quad d_2+b_2n_1>-n_2-1. \label{eq:reg-2}
\end{align}
\end{prop}

\begin{proof}
By Proposition~2.38 of \cite{aprodu10} it is enough to show that $H^i(\P^{\nn}, \O_{\P^{\nn}}(\dd+\bb(|\nn|-i)))=0$ for all $i>0$. Using the K\"{u}nneth formula \cite[\href{https://stacks.math.columbia.edu/tag/0BEC}{Tag 0BEC}]{stacks-project} to compute $H^i(\P^{\nn}, \O_{\P^{\nn}}(\dd+\bb(|\nn|-i)))=0$ we see that these cohomology groups are only potentially non-zero when $i=n_1,n_2,$ and $|\nn|$. In particular, $H^{n_1}(P^{\nn}, \O_{\P^{\nn}}(\dd+\bb(|\nn|-n_1))$ is isomorphic to $H^{n_1}(\P^{n_1},\O_{\P^{n_1}}(d_1+b_1n_2))\otimes H^{0}(\P^{n_2}, \O_{\P^{n_2}}(d_2+b_2n_2))$. Thus, the condition that $H^{n_1}(P^{\nn}, \O_{\P^{\nn}}(\dd+\bb(|\nn|-n_1))=0$ is equivalent to \eqref{eq:reg-1}. An analogous argument shows that the vanishing of $H^{n_2}(P^{\nn}, \O_{\P^{\nn}}(\dd+\bb(|\nn|-n_2))$ is equivalent to \eqref{eq:reg-2}. Finally, in the last case, when $i=|\nn|$, by using the K\"{u}nneth formula we see that $H^{|\nn|}(\P^{\nn}, \O_{\P^{\nn}}(\dd+\bb(|\nn|-|\nn|)))$ is zero if and only if $d_1>- n_1-1$ or $d_2>-n_2-1$. Since $d_1\geq1$ and $d_2\geq1$ these conditions are always satisfied. 
\end{proof}

%%%%%%%%%%%%%%%%%%%%%%%%%%%%%%%%%%%%%%%%%%%%%%%%%%%%%%%%%%%%%%%%%%%%%%%%%%%%%%%%%%%%%%%%%%%%%%%%%%%%%%%%%%%%%%%%
 \section{A Regular Sequence on $\P^{n_1}\times \P^{n_2}$}\label{sec:reg-seq}
%%%%%%%%%%%%%%%%%%%%%%%%%%%%%%%%%%%%%%%%%%%%%%%%%%%%%%%%%%%%%%%%%%%%%%%%%%%%%%%%%%%%%%%%%%%%%%%%%%%%%%%%%%%%%%%%

One useful approach when attempting to construct non-zero syzygies is to quotient by a linear regular sequence as this does not change the Koszul cohomology groups \cite[Theorem~2.20]{aprodu10}. For example, in order to prove non-vanishing results for $\P^{n}$, Ein, Erman, and Lazarsfeld quotient by the regular sequence consisting of powers of the variables \cite{einErmanLazarsfeld16}. Since we are working on a product of projective spaces such a regular sequence is not an option for us. Namely there are no monomial regular sequences of bi-degree $\dd$ of length $|\nn|+1$ on either the Cox ring of $\P^{\nn}$ or the homogeneous coordinate ring of $\P^{\nn}$ embedded by $\O_{\P^{\nn}}(\dd)$. 

Instead, we choose to work with a sequence of multigraded forms which form a virtual regular sequence of length $|\nn|+1$ on the Cox ring of $\P^{\nn}$ (i.e. $S$). That is to say a sequence of elements $g_{0},g_{1},\ldots,g_{|\nn|}$ of bi-degree $\dd$ whose support is contained in the irrelevant ideal $\langle x_0,x_1,\ldots,x_{n_1}\rangle \cap \langle y_0,y_1,\ldots,y_{n_2}\rangle$ of $\P^{\nn}$. Since the ideal $\langle g_0,g_1,\ldots,g_{|\nn|}\rangle$ is supported on the irrelevant ideal, the $g_i$ form a regular sequence on $\P^{\nn}$. By the isomorphism between $S_{\dd}$ and $R_{1}$ discussed in the previous section, such $g_{0},g_{1},\ldots,g_{|\nn|}$ correspond to a sequence of linear forms $\ell_{0},\ell_{1},\ldots,\ell_{|\nn|}$ in $R$, that is a regular sequence on $S(\bb;\dd)$.The $g_{i}$ we use generalize forms first introduced by Eisenbud and Schreyer in \cite{eisenbudSchreyer09}, and later used in \cite{berkesch13} and \cite{oeding17}.

\begin{defn}
Fix $\nn=(n_1,n_2)\in \Z^2_{\geq1}$ and $\dd=(d_1,d_2)\in \Z^2_{\geq1}$. Given $0\leq t \leq |\nn|$ we define
\[
g_t\coloneqq\sum_{\substack{i+j=t\\0\leq i \leq n_1\\ 0 \leq j \leq n_2}}x_{i}^{d_1}y_j^{d_2}.
\] 
\end{defn}

\begin{example}
For example if $\nn=(1,1)$ and $\dd=(d_1,d_2)$ then there are three $g_{t}$'s:
\[
g_0=x_0^{d_1}y_0^{d_2}, \quad g_1=x_0^{d_1}y_1^{d_2}+x_1^{d_1}y_0^{d_2}, \quad g_2=x_1^{d_1}y_1^{d_2}.
\]
On the other hand if $\nn=(2,3)$ and $\dd=(d_1,d_2)$ there are six $g_{t}$'s:
\[
g_0=x_0^{d_1}y_0^{d_2}, \quad g_1=x_0^{d_1}y_1^{d_2}+x_1^{d_1}y_0^{d_2}, \quad g_2=x_0^{d_1}y_2^{d_2}+x_1^{d_1}y_1^{d_2}+x_{2}^{d_1}y_{0}^{d_2},
\]
\[
g_{3}=x_{0}^{d_1}y_{3}^{d_2}+x_{1}^{d_1}y_{2}^{d_2},
\quad g_{4}=x_{1}^{d_1}y_{3}^{d_2}+x_{2}^{d_1}y_{2}^{d_2}, \quad g_{5}=x_{2}^{d_1}y_{3}^{d_2}.
\]
\end{example}

\begin{defn}
Throughout the paper we let $\R(\nn,\dd)=\langle g_0,g_1,\ldots,g_{|\nn|}\rangle$. Note that $\R(\nn,\dd)$ depends on both $\nn=(n_1,n_2)$ and $\dd=(d_1,d_2)$, however, for notational hygiene we often suppress this and simply write $\R$ or $\R(\dd)$ for $\R(\nn,\dd)$ when we feel it will not cause confusion.  We denote the quotient $S/\R$ by $\overline{S}$.  
\end{defn}

An extremely important aspect of these particular forms is that they behave nicely when quotienting by $x_{n_1}$ or $y_{n_2}$. For example, if $\nn=(2,3)$ then the image of $g_{2}=x_0^{d_1}y_2^{d_2}+x_1^{d_1}y_1^{d_2}+x_{2}^{d_1}y_{0}^{d_2}$ in $S/\langle x_{2}\rangle$ is $x_0^{d_1}y_2^{d_2}+x_1^{d_1}y_1^{d_2}$, which is the same as $g_{2}$ when $\nn=(1,3)$. This makes them amenable to inductive arguments on $n_1$ or $n_2$. We make significant use of this fact throughout, and so record it in the following lemma.

\begin{lemma}\label{lem:induction-ideal}
Fix $\nn=(n_1,n_2)\in \Z_{\geq1}^2$ and $\dd=(d_1,d_2)\in\Z^2_{\geq1}$. Let $S'=\K[x_0,x_1,\ldots, x_{n_1-i},y_{0},y_1,\ldots,y_{n_2-j}]$ considered with the bi-grading given by $\deg x_s=(1,0)\in \Z^2$ and $\deg y_t=(0,1)\in \Z^2$. Considering $\R(n_1-i,n_2-j,\dd)\subset S'$ there exists a natural isomorphism
\begin{center}
\begin{tikzcd}[column sep = 5em]
\frac{\displaystyle S'}{\displaystyle\R(n_1-i,n_2-j,\dd)} \arrow[r, leftrightarrow, "\sim"] & \frac{\displaystyle S}{\displaystyle\R(\nn,\dd)+\langle x_{n_1-i+1},x_{n_1-i+2},\ldots,x_{n_1},y_{n_2-j+1},y_{n_2-j+2},\ldots,y_{n_2}\rangle}.
\end{tikzcd}
\end{center}
\end{lemma}

\begin{proof}
By induction it is enough to consider the case when $i=1$ and $j=0$. There is a natural isomorphism 
\begin{center}
\begin{tikzcd}[column sep = 5em]
\frac{\displaystyle \frac{S}{\langle x_{n_1}\rangle}}{\displaystyle \frac{\R(\nn,\dd)+\langle x_{n_1}\rangle}{\langle x_{n_1}\rangle}}\arrow[r, leftrightarrow, "\sim"] & \frac{\displaystyle S}{\displaystyle\R(\nn,\dd)+\langle x_{n_1}\rangle},
\end{tikzcd}
\end{center}
and since $S/\langle x_{n_1}\rangle\cong S'$ it is enough to show that $\frac{\R(\nn,\dd)+\langle x_{n_1}\rangle}{\langle x_{n_1}\rangle}$ is isomorphic to $\R(n_1-1,n_2,\dd)$. A straightforward argument shows that $\frac{\R(\nn,\dd)+\langle x_{n_1}\rangle}{\langle x_{n_1}\rangle}$ is isomorphic to the ideal $\langle \overline{g}_{0},\overline{g}_{1},\ldots,\overline{g}_{|\n|}\rangle$ where $\overline{g}_{t}$ is the image of $g_{t}$ in $S/\langle x_{n_1}\rangle$. However, one sees that
\[
\overline{g}_{t}=\sum_{\substack{a+b=t\\ 0\leq a \leq n_1-1 \\ 0 \leq b \leq n_2}} x_{a}^{d_1}y_{b}^{d_2}, 
\]
and so considered as an element of $S'$, the ideal $\langle \overline{g}_{0},\overline{g}_{1},\ldots,\overline{g}_{|\nn|}\rangle$ is isomorphic to $\R(n_1-1,n_2,\dd)$. 
\end{proof}

As noted in the previous section, there is a natural isomorphism between $R_{1}$ and $S_{\dd}$, and we write $\ell_{t}$ for the image of $g_{t}$ in $R_{1}$ under this isomorphism. Notice that while $g_{t}\in S$ has bi-degree $\dd$, the corresponding element $\ell_{t}\in R$ is of degree one. We then let $\L(\nn,\dd)$ be the ideal $\langle \ell_{0},\ell_{1},\ldots,\ell_{|\nn|}\rangle \subset R$. As with $\R(\nn,\dd)$ we will often write $\L$ or $\L(\dd)$ for $\L(\nn,\dd)$ when $\nn$ and $\dd$ are clear from context. We write $\overline{R}$ for the quotient $R/\L$, and $\overline{S}(\bb;\dd)$ for $S(\bb;\dd)/\L S(\bb;\dd)$, which we consider as a $\overline{R}$-module. The natural isomorphisms discussed in the previous section remain isomorphisms after quotienting by $\R$ and $\L$
\begin{center}
\begin{tikzcd}[column sep = 3.5em]
\overline{R}_{1}\arrow[r, leftrightarrow, "\sim"] & \overline{S}(\zero;\dd)_{1} \arrow[r, leftrightarrow, "\sim"] & \overline{S}_{\dd}.
\end{tikzcd}
\end{center}

As indicated in the start of the section these $\ell_{t}$'s form a regular sequence on $S(\bb;\dd)$ as long as $S(\bb;\dd)$ is Cohen-Macaulay as an $R$-module. The case when $\dd=\one$ and $\bb=\zero$ was shown by Eisenbud and Schreyer in their work on Boij-S\"{o}derberg theory \cite[Proposition~5.2]{eisenbudSchreyer09}. The following proposition generalizes their argument to the case of all $\dd\in \Z^{t}_{\geq1}$.

\begin{prop}\label{prop:regular-sequence}
Fix $\nn=(n_1,n_2)\in \Z_{\geq1}^2$, $\dd=(d_1,d_2)\in\Z^2_{\geq1}$, and $\bb=(b_1,b_2)\in\Z^2$. If 
\[
\frac{d_1}{d_2}b_2-b_1<n_1+1\quad \quad \text{and} \quad \quad \frac{d_2}{d_1}b_1-b_2<n_2+1
\]
then the forms $\ell_{0},\ell_{1},\ldots,\ell_{|\nn|}$ are a regular sequence on $S(\bb;\dd)$ as an $R$-module. 
\end{prop}

A key part of the Proposition~\ref{prop:regular-sequence} is the following characterization of when $S(\bb;\dd)$ is Cohen-Macaulay as an $R$-module. In particular, the inequalities appearing in Proposition~\ref{prop:regular-sequence} are needed as they exactly describe when $S(\bb;\dd)$ is Cohen-Macaulay as an $R$-module. This is a major difference between a product of projective spaces and a single projective space, as in the case of a single projective space the equivalent of $S(\bb;\dd)$ is always Cohen-Macaulay \cite{einErmanLazarsfeld16}. 

\begin{prop}\label{prop:cohen-macaulay}
Fix $\nn=(n_1,n_2)\in \Z_{\geq1}^2$, $\dd=(d_1,d_2)\in\Z^2_{\geq1}$, and $\bb=(b_1,b_2)\in\Z^2$. $S(\bb;\dd)$ is Cohen-Macaulay as an $R$-module if and only if:
\[
\frac{d_1}{d_2}b_2-b_1<n_1+1\quad \quad \text{and} \quad \quad \frac{d_2}{d_1}b_1-b_2<n_2+1.
\]
\end{prop}

Note that $S(\bb;\dd)$ is Cohen-Macaulay for all $\dd$ if $\bb=\zero$. In particular, since a product of projective spaces is a smooth toric variety, the case when $\bb=0$ follows from a far more general result of Hochster \cite{hochster72} (see also \cite[Theorem 9.2.9]{coxLittleSchenck11}). 

\begin{proof}[Proof of Proposition~\ref{prop:cohen-macaulay}]
If we write $H_{R_+}^i(S(\bb;\dd))$ for the $i$-th local cohomology module of $S(\bb;\dd)$, then $S(\bb;\dd)$ is Cohen-Macaulay if and only if $\dim S(\bb;\dd)$ is equal to $\inf\{i\in \N \; : \;H^i_{R_+}(S(\bb;\dd))\neq0\}$ \cite[Theorem~9.1]{twentyFourHours}. Moreover, since $S(\bb;\dd)$ is isomorphic to the section ring of $(\iota_{\dd})_* \O_{\P^{\n}}(\bb)$ where $\iota_{\dd}:\P^{\nn}\rightarrow{}\P^{r_{\nn,\dd}}$ is the $\dd$'uple Segre-Veronese map induced by $\O_{\P^{\nn}}(\dd)$
\[
\inf\left\{i\in \Z_{>1} \; \big| \;H^i_{R_+}(S(\bb;\dd))\neq0\right\}=\inf\left\{i\in\Z_{>1} \; \bigg| \; \begin{matrix}
H^{i-1}\left(\P^{\rr_{\nn,\dd}}, (\iota_{\dd})_* \O_{\P^{\nn}}(\bb)(k)\right)\neq0\\
\text{for some $k\in\Z$}
\end{matrix}
\right\},
\]
and so $S(\bb;\dd)$ is Cohen-Macaulay if and only if $(\iota_{\dd})_* \O_{\P^{\nn}}(\bb)$ has no intermediate cohomology \cite[Theorem~13.21]{twentyFourHours}. Since $H^{i-1}\left(\P^{\rr_{\nn,\dd}}, (\iota_{\dd})_* \O_{\P^{\nn}}(\bb)(k)\right)$ is isomorphic to $H^{i-1}\left(\P^{\nn}, \O_{\P^{\nn}}(\bb+k\dd)\right)$ by the K\"{u}nneth formula \cite[\href{https://stacks.math.columbia.edu/tag/0BEC}{Tag 0BEC}]{stacks-project} we further reduce to cohomology computation on $\P^{n_1}$ and $\P^{n_2}$. From this we see that there is no intermediate cohomology if for every $k\in \Z$: 
\begin{align}
b_1+kd_1>-(n_1+1) \quad \quad \text{or} \quad \quad b_2+kd_2<0 \label{eq:line1} \\ 
\intertext{and}
b_2+kd_2>-(n_2+1) \quad \quad \text{or} \quad \quad b_1+kd_1<0. \label{eq:line2}
\end{align}
Now note the first inequality in Equation~\eqref{eq:line1} is true for every $k>-(n_1+1+b_1)/d_1$ while the second is true for every $k<-b_2/d_2$. Thus, Equation~\eqref{eq:line1} is true for all $k$ if and only if $(n_1+1+b_1)/d_1<-b_2/d_2$. Rearranging this inequality gives the first hypothesis in the proposition statement. A similar analysis for Equation~\eqref{eq:line2} produces the second hypothesis.
\end{proof}

\begin{proof}[Proof of Proposition~\ref{prop:regular-sequence}]
By Proposition~\ref{prop:cohen-macaulay} $S(\bb;\dd)$ is Cohen-Macaulay as an $R$-module, and so showing that $\ell_{0},\ell_{1},\ldots,\ell_{|\nn|}$ is a regular sequence on $S(\bb;\dd)$ is equivalent to showing that $\ell_{0},\ell_{1},\ldots,\ell_{|\nn|}$ is part of a system of parameters. Equivalently that $\dim \overline{S}(\bb;\dd) =\dim S(\bb;\dd)-(|\nn|+1)$ \cite[Theorem~2.12]{burnsHerzog93}. Being system of parameters is a set-theoretic condition on the support of $\overline{S}(\bb;\dd)$, and since 
\[
\supp_{\overline{R}}\overline{S}(\bb;\dd)=\supp_{R} S(\bb;\dd)\otimes_R \overline{R}=\supp_{R} S(\bb;\dd)\otimes_{R}\otimes \overline{S}(0;\dd)\subset \supp_{R} S(\bb;\dd)\cap \supp \overline{S}(0;\dd)
\] we may reduce to the case when $\bb=\zero$. Now let $\cI(\dd)$ be the ideal sheaf generated by $\ell_{0},\ell_{1},\ldots,\ell_{|\nn|}$. Considering the map:
\begin{center}
\begin{tikzcd}[row sep = .75 em, column sep = 3em]
\P^{\nn}\rar{\psi}&\P^{\nn} \\ 
x_{i,j}\rar[mapsto]&x_{i,j}^{d_i}
\end{tikzcd}
\end{center}
one sees that $\psi^{*}\cI(\dd)=\cI(\one)$. Therefore, since $H^0(\cI(\dd))=\L(\nn,\dd)$ we see that we may further reduce to the case when $\dd=\one$. This case was proven in \cite[Proposition~5.2]{eisenbudSchreyer09}. 
\end{proof}

Since $\L$ is generated by a linear regular sequence on $S(\bb;\dd)$, quotienting by $\L$ does not change the cohomology of the Koszul complex of \eqref{eq:kozul-S}. 

\begin{notation}
Fix $\nn=(n_1,n_2)\in \Z_{\geq1}^2$, $\dd=(d_1,d_2)\in\Z^2_{\geq1}$, and $\bb=(b_1,b_2)\in\Z^2$. We let $K^{\overline{R}}_{p,q}(\overline{S}(\bb;\dd))$ denote the cohomology of the following chain complex
\begin{equation}
\begin{tikzcd}[column sep = 3em]
\cdots \rar{}& \Alt^{p+1}\overline{S}_{\dd}\otimes \overline{S}_{(q-1)\dd+\bb}\rar{\overline{\partial}_{p+1}}&\Alt^{p}\overline{S}_{\dd}\otimes \overline{S}_{q\dd+\bb}\rar{\overline{\partial}_p}&\Alt^{p-1}\overline{S}_{\dd}\otimes \overline{S}_{(q+1)\dd+\bb}\rar{}&\cdots
\end{tikzcd}.
\end{equation}
\end{notation}

\begin{cor}\label{cor:artinian-reduction}
Fix $\nn=(n_1,n_2)\in \Z_{\geq1}^2$, $\dd=(d_1,d_2)\in\Z^2_{\geq1}$, and $\bb=(b_1,b_2)\in\Z^2$.  If 
\[
\frac{d_1}{d_2}b_2-b_1<n_1+1\quad \quad \text{and} \quad \quad \frac{d_2}{d_1}b_1-b_2<n_2+1
\]
then for all $p,q\in \Z_{\geq0}$ there exists a natural isomorphism 
\begin{equation*}
\begin{tikzcd}[column sep = 5em]
K_{p,q}(\nn,\bb;\dd) \rar[leftrightarrow]{\sim} & K^{\overline{R}}_{p,q}(\overline{S}(\bb;\dd))
\end{tikzcd}
\end{equation*}
\end{cor}

\begin{proof}
Combine Proposition~\ref{prop:regular-sequence} with Theorem~2.20 from \cite{aprodu10}.
\end{proof}

%%%%%%%%%%%%%%%%%%%%%%%%%%%%%%%%%%%%%%%%%%%%%%%%%%%%%%%%%%%%%%%%%%%%%%%%%%%%%%%%%%%%%%%%%%%%%%%%%%%%%%%%%%%%%%%%
\section{Ideal Membership for $\R$}\label{sec:ideal-membership}
%%%%%%%%%%%%%%%%%%%%%%%%%%%%%%%%%%%%%%%%%%%%%%%%%%%%%%%%%%%%%%%%%%%%%%%%%%%%%%%%%%%%%%%%%%%%%%%%%%%%%%%%%%%%%%%%

In this section, we turn our attention to describing when certain monomials are contained in the ideal $\R$ introduced in Section~\ref{sec:reg-seq}. This highlights a significant challenge when generalizing the work of Ein, Erman, and Lazarsfeld from the case of a single projective space to a product of projective spaces. Namely, since there are no monomial regular sequences of length $|\nn|+1$ on $\P^{\nn}$, we must work with a regular sequence for which, the ideal membership question is more difficult. For example, describing when a given element of $S$ is contained in $\R=\langle g_0,g_1,\ldots,g_{|\nn|}\rangle$, is more complicated then determining when an element is in $\langle x_0^{d},x_{1}^{d},\ldots,x_{n}^{d}\rangle\subset \K[x_0,x_1,\ldots,x_n]$. This section is dedicated to studying the ideal membership question for $\R$. 

Our approach to ideal membership for $\R$ is to make use of the fact that $\R$ is homogeneous with respect to a number of interesting gradings. For example, in Section~\ref{sec:mod-deg} we introduce the notion of the modular degree of an element of $S$. This induces a $\left(\Z/\langle d_1\rangle\right)^{\oplus n_1+1}\oplus\left(\Z/\langle d_2\rangle\right)^{\oplus n_2+1}$-grading on $S$ that $\R$ is homogeneous with respect to. Using this grading we show that the ideal membership question for $\R(\nn,\dd)$ can be reduced to the ideal membership question for $\R(\nn,\one)$. 

Having reduced the question of ideal membership to the case when $\dd=\one$, we then introduce the notion of the index weighted degree, which induces a non-standard $\Z$-grading on $S$. The index weighted grading allows us to discuss the $\K$-vector space $S_{\aa,k}$ spanned by monomials of bi-degree $\aa$ and index weighted degree $k$. Using this refinement together with a series of spectral sequence arguments we gain insight into the ideal membership question for $\R(\nn,\one)$. For example, we prove the following:

\begin{theorem}\label{thm:tri-deg-vanishing}
Fix $\nn=(n_1,n_2)\in \Z^2_{\geq1}$, $\aa=(a_1,a_2)\in \Z^2_{\geq0}$, and $k\in\Z_{\geq0}$. If $\aa$ and $k$ satisfy one of the following inequalities:
\begin{enumerate}
\item $a_1\geq1$ and  $a_2\geq n_1+1$,
\item $a_2\geq1$ and $a_1\geq n_{2}+1$,
\item $0\leq k\leq a_1a_2-1$, or
\end{enumerate}
then $S_{\aa,k}=\R(\one)_{\aa,k}$. Moreover, if $k=a_1a_2$ then $\dim S_{\aa,k}=\dim \R(\one)_{\aa,k}-1$.
\end{theorem}

Combining these arguments gives us a detailed understanding of what monomials are in $\R(\dd)$. This, in turn, allows us to understand the ideal quotient $(\R:_{S}f)$ for particular polynomials $f\in S$, and this provides the range of non-vanishing appearing in Theorem~\ref{thm:main}. 

\subsection{The Modular Degree on $S$}\label{sec:mod-deg}

Throughout this section given $\vv=(v_1,v_2\ldots,v_n)\in \Z^{n}$ and $d\in \Z$ we write $\vv\bmod d$ to mean $(v_1\bmod d,v_2\bmod d,\ldots,v_n\bmod d)\in \left(\Z/\langle d\rangle \right)^{n}$. Further we use the following multi-index notation for monomials in $S$ and $\overline{S}$.

\begin{notation}
Given $\vv=(v_0,v_1,\ldots,v_{n_1})\in \Z^{n_1+1}_{\geq 0}$ and $\ww=(w_0,w_1,\ldots,w_{n_2})\in \Z^{n_2+1}_{\geq 0}$ write $\xx^{\vv}\yy^{\ww}$ for the monomial:
\[
\xx^{\vv}\yy^{\ww}=\prod_{i=0}^{n_1}x_i^{v_i}\prod_{j=0}^{n_2}y_j^{w_j}\in S.
\]
\end{notation}

With this notation in hand, we define the modular degree of a monomial in $S$ as follows.

\begin{defn}
Fix $\nn=(n_1,n_2)\in \Z^2_{\geq1}$ and $\dd=(d_1,d_2)\in \Z^2_{\geq1}$. Given a monomial $\xx^{\vv}\yy^{\ww}\in S$ we define its modular degree to be:
\[
\moddeg\left(\xx^{\vv}\yy^{\ww}\right)=(\vv\bmod d_1,\ww\bmod d_2)\in \left(\Z/\langle d_1\rangle\right)^{\oplus n_1+1}\oplus\left(\Z/\langle d_2\rangle\right)^{\oplus n_2+1}.
\]
\end{defn}

Immediate from the definition we see that the modular degree induces a $\left(\Z/\langle d_1\rangle\right)^{\oplus n_1+1}\oplus\left(\Z/\langle d_2\rangle\right)^{\oplus n_2+1}$-grading on $S$ as follows
\[
S\cong \bigoplus_{\alpha \in \left(\Z/\langle d_1\rangle\right)^{\oplus n_1+1}\oplus\left(\Z/\langle d_2\rangle\right)^{\oplus n_2+1}} S_{\alpha}
\]
where $S_{\alpha}$ is the $\K$-vector space spanned by monomials $m\in S$ such that $\moddeg(m)=\alpha$. We call this the modular grading, and $\R$ is homogeneous with respect to it. 

\begin{lemma}\label{lem:mod-deg-homg}
Fix $\nn=(n_1,n_2)\in \Z^2_{\geq1}$ and $\dd=(d_1,d_2)\in \Z^2_{\geq1}$. The modular degree gives $S$ a $\left(\Z/\langle d_1\rangle\right)^{\oplus n_1+1}\oplus\left(\Z/\langle d_2\rangle\right)^{\oplus n_2+1}$-grading. Moreover, the ideal $\R$ is homogeneous with respect to this grading.
\end{lemma}

\begin{proof}
To show that the modular degree induces a grading on $S$ it is enough to show that if $m,m'\in S$ are monomials then $\moddeg(m\cdot m')$ is equal to $\moddeg(m)+\moddeg(m')$. This follows from the fact that $
\Z/\langle d_i\rangle$ is an abelian group. 

Shifting to showing that $\R=\langle g_{0},g_{1},\ldots,g_{|\nn|}\rangle$ is homogeneous with respect to this grading it is enough to show that each of the generators are homogeneous. Towards this recall that 
\[ 
g_{t}=\sum_{\substack{i+j=t\\0\leq i \leq n_1\\ 0 \leq j \leq n_2}}x_{i}^{d_1}y_j^{d_2},
\] 
and so each term in $g_{t}$ has modular degree $\zero$ meaning $g_{t}$ is homogeneous with respect to this grading.
\end{proof}

The key property of the modular grading on $S$ is that thinking of $S^{[\dd]}\coloneqq \K[x_0^{d_1},x_1^{d_1},\ldots,x_{n_{1}}^{d_{1}},y_{0}^{d_{2}},y_{1}^{d_{2}},\ldots,y_{n_2}^{d_{2}}]$ as a sub-ring of $S$ then $S_{\alpha}$ is a free rank one $S^{[\dd]}$-module for every $\alpha \in \left(\Z/\langle d_1\rangle\right)^{\oplus n_1+1}\oplus\left(\Z/\langle d_2\rangle\right)^{\oplus n_2+1}$. Given a monomial $\xx^{\vv}\yy^{\ww}\in S_{\alpha}$ then by the division algorithm we may write $v_i = q_{i}d_1+r_{i}$ and $w_i = q'_{i}d_2+r'_2$ where $0\leq r_i < d_1$ and $0\leq r'_i < d_2$. This allows us to write $\xx^{\vv}\yy^{\ww}$ as
\[
\xx^{\vv}\yy^{\ww} = \underbrace{\left(\prod_{i=0}^{n_1}x_i^{r_i}\prod_{i=0}^{n_2}y_i^{r'_i}\right)}_{\text{I}} \underbrace{\left(\prod_{i=0}^{n_1}x_i^{q_{i}d_1}\prod_{i=0}^{n_2}y_{i}^{q'_{i}d_2}\right)}_{\text{II}},
\]
where II is a monomial in $S^{[\dd]}$ and I is a monomial determined entirely by the modular degree of $\xx^{\vv}\yy^{\ww}$. 

\begin{defn}
Given a monomial $\xx^{\vv}\yy^{\ww}\in S$  we define $\remd(\xx^{\vv}\yy^{\ww})$ to be the monomial
\[
\remd\left(\xx^{\vv}\yy^{\ww}\right)=\left(\prod_{i=0}^{n_1}x_i^{r_i}\prod_{i=0}^{n_2}y_i^{r'_i}\right)\in S,
\]
where by the division algorithm $v_i = q_{i}d_1+r_{i}$ and $w_i = q'_{i}d_2+r'_2$ with $0\leq r_i < d_1$ and $0\leq r'_i < d_2$ .
\end{defn}

\begin{example}
Let $\nn=(1,2)$ so that $S=\K[x_0,x_1,y_0,y_1,y_2]$ and set $\dd=(3,5)$. If $f=x_{0}^{5}x_1^{3}y_{0}^{5}y_{1}^{11}y_{2}^{8}$ then the modular degree of $f$ is $((2,0),(1,1,3))$ and $\remd(f)=x_0^2y_0y_1y_2^3$. Any element of modular degree $((2,0),(1,1,3))$ can be written as $x_0^2y_0y_1y_2^3$ times an element of $S^{[\dd]}$. For example, $g=x_0^{11}y_0y_1^{6}y_2^{103}$ also has modular degree $((2,0),(1,1,3))$ , and we may write it as $\remd(f)\cdot x_0^{3\cdot 3}y_1^{5}y_2^{20\cdot 5}$.
\end{example}

\begin{lemma}
Fix $\nn=(n_1,n_2)\in \Z^2_{\geq1}$ and $\dd=(d_1,d_2)\in \Z^2_{\geq1}$. For any $\alpha\in  \left(\Z/\langle d_1\rangle\right)^{\oplus n_1+1}\oplus\left(\Z/\langle d_2\rangle\right)^{\oplus n_2+1}$ the vector space $S_{\alpha}$ is a free rank one $S^{[\dd]}$-module, which is generated by $\remd(f)$ for any $f\in S_{\alpha}$.
\end{lemma}

\begin{proof}
First let us check that $S_{\alpha}$ has the structure of a $S^{[\dd]}$-module. This amounts to showing that $S_{\alpha}$ is closed under multiplication by elements in $S^{[\dd]}$. Since $S^{[\dd]}$ is generated by monomials of bi-degree $\dd$ we further reduce to showing $S_{\alpha}$ is closed under multiplication by $x_{i}^{d_1}$ and $y_j^{d_2}$ for $i=0,1,\ldots,n_1$ and $j=0,1,\ldots,n_2$. This follows immediately from the definition of the modular degree.

Turing our attention to showing that $S_{\alpha}$ is free of rank one fix $\alpha\in  \left(\Z/\langle d_1\rangle\right)^{\oplus n_1+1}\oplus\left(\Z/\langle d_2\rangle\right)^{\oplus n_2+1}$, let $\xx^{\vv}\yy^{\vv}\in S_{\alpha}$ be a monomial. By the division algorithm we may write $v_i = q_{i}d_1+r_{i}$ and $w_i = q'_{i}d_2+r'_2$ where $0\leq r_i < d_1$ and $0\leq r'_i < d_2$. One readily checks that
\[
\xx^{\vv}\yy^{\ww} = \left(\prod_{i=0}^{n_1}x_i^{r_i}\prod_{i=0}^{n_2}y_i^{r'_i}\right) \left(\prod_{i=0}^{n_1}x_i^{q_{i}d_1}\prod_{i=0}^{n_2}y_{i}^{q'_{i}d_2}\right)=\remd\left(\xx^{\vv}\yy^{\ww}\right)\left(\prod_{i=0}^{n_1}x_i^{q_{i}d_1}\prod_{i=0}^{n_2}y_{i}^{q'_{i}d_2}\right).
\]
Moreover, since $r_{i}$ and $r'_{i}$ determine the modular degree of $\xx^{\vv}\yy^{\ww}$ we see that every monomial in $S_{\alpha}$ is of the form $m\cdot\remd(\xx^{\vv}\yy^{\ww})$ for a unique $m\in S^{[\dd]}$. 
\end{proof}

We now state a few basic properties regarding $\remd(f)$ that follows immediately from the previous lemma. 

\begin{lemma}\label{lem:round-down-properties}
Fix $\nn=(n_1,n_2)\in \Z^2_{\geq1}$ and $\dd=(d_1,d_2)\in \Z^2_{\geq1}$. If $f,g\in S$ are homogeneous with respect to the modular grading then
\begin{enumerate}
\item $f$ is divisible by $\remd(f)$,
\item $f/\remd(f) \in S^{[\dd]}$, and
\item $\moddeg f = \moddeg g$ if and only if $\remd(f) = \remd(g)$.
\end{enumerate}
\end{lemma}

Finally, the following proposition shows how the modular grading can be used to reduce the ideal membership question for $\R(\dd)$ to the ideal membership question for $\R(\one)$. Before stating it, however, we fix the following notation that given a monomial $m\in S^{[\dd]}$ we let $m^{1/\dd}$ be the monomial that is the image of $m$ under the isomorphism:
\[
\begin{tikzcd}
S^{[\dd]}\arrow[rr,"^{1/\dd}"]& &S & & x_i^{d_1}\rar[mapsto] &x_i & y_i^{d_1}\rar[mapsto] & y_i 
\end{tikzcd}.
\]

\begin{prop}\label{prop:containment}
Fix $\nn=(n_1,n_2)\in \Z^2_{\geq1}$ and $\dd=(d_1,d_2)\in \Z^2_{\geq1}$. Let $f\in S$ be homogeneous with respect to the modular grading then $f\in \R(\dd)$ if and only if $\left( f/\remd(f)\right)^{1/\dd} \in \R(\one)$.
\end{prop}

\begin{proof}
By definition $f\in \R(\dd)$ if and only if there exists $h_i\in S$ such that:
\[
f=\sum_{i=0}^{|\nn|} h_ig_i.
\]
Now since $f$ is homogeneous with respect to the modular grading without loss of generality we may assume that the $h_i$ are also homogeneous with respect to the modular grading. Moreover, by Lemma~\ref{lem:mod-deg-homg} the modular degree of $g_{t}$ is equal to $\zero$, and so $\moddeg f=\moddeg h_i$. In particular, by part (3) of Lemma~\ref{lem:round-down-properties} we know that $\remd(f) = \remd(h_i)$. By part (1) Lemma~\ref{lem:round-down-properties} we know that $f$ is divisible by $\remd(f)$, and so combining these we have that $f\in \R(\dd)$ if and only if:
\[
\frac{f}{\remd(f)}=\sum_{i=0}^{|\nn|} \frac{h_ig_i}{\remd(f)}=\sum_{i=0}^{|\nn|} \frac{h_i}{\remd(h_i)}g_i.
\]
By part (2) of Lemma~\ref{lem:round-down-properties} the above relation is actually a relation in the subring $S^{[\dd]}$. Since under the isomorphism $-^{1/\dd}$ the image of $\R(\dd)$ is $\R(\one)$, we see that $f\in \R(\dd)$ if and only if $\left( f/\remd(f) \right)^{1/\dd} \in \R(\one)$. \end{proof}

\subsection{The Index Weighted Degree}

If we look at one of the generators of $\R(\dd)$ we see that the lower indices of each term all sum to the same thing. For example, $g_{1}=x_{0}^{d_1}y_1^{d_2}+x_{1}^{d_1}y_0^{d_2}$ and the lower indices of each term sum to one. We exploit this symmetry by introducing a non-standard $\Z$-grading on $S$, which we call the index weighted grading, which $\R$ is homogeneous with respect to. Using this grading we will prove Theorem~\ref{thm:tri-deg-vanishing}, and state a conjecture describing exactly when $S_{\aa,k}=\R(\one)_{\aa,k}$. 

\begin{defn}
Fix $\nn=(n_1,n_2)\in \Z^2_{\geq1}$ and $\dd=(d_1,d_2)\in \Z^2_{\geq1}$. The index weighted grading on $S$ is the non-standard $\Z$-grading given by letting $\indeg x_i=d_2i$ and $\indeg y_j=d_1j$ for $i=0,1,\ldots,n_1$ and $j=0,1,\ldots,n_2$.
\end{defn}

The important property of the index weighted grading is that $\R$ is homogeneous with respect to it.

\begin{lemma}\label{lem:homogeneous-wrt-idex}
Fix $\nn=(n_1,n_2)\in \Z^2_{\geq1}$ and $\dd=(d_1,d_2)\in \Z^2_{\geq1}$. The ideal $\R$ is homogeneous with respect to the index weighted grading.
\end{lemma}

\begin{proof}
Recall that $\R=\langle g_0,g_{1}\ldots,g_{|\nn|}\rangle$ where for $0\leq t \leq |\nn|$:
\[ 
g_{t}=\sum_{\substack{i+j=t\\0\leq i \leq n_1\\ 0 \leq j \leq n_2}}x_{i}^{d_1}y_j^{d_2}.
\] 
Suppose $x_{i}^{d_1}y_j^{d_2}$ is a term appearing in $g_{t}$ so that $i+j=t$. Now we have that:
\[
\indeg\left(x_{i}^{d_1}y_j^{d_2}\right)=d_1d_2i+d_1d_2j=d_1d_2(i+j)=d_1d_2t,
\]
and so each term of $g_{t}$ has the same index weighted degree meaning $\R$ is homogeneous.\end{proof}

\begin{defn}
Given $\aa\in \Z^2$ and $k\in \Z$ we write $S_{\aa,k}$ (respectively $\overline{S}_{\aa,k}$ and $\R_{\aa,k}$) for the $\K$-vector space spanned by monomials in $S$ (respectively $\overline{S}$ and $I$) of bi-degree $\aa$ and index weighted degree $k$.
\end{defn}

The following conjecture describes exactly the $\aa\in\Z^2$ and $k\in \Z$ for which $\R(\one)_{\aa,k}$ is equal to $S_{\aa,k}$. Combined with Proposition~\ref{prop:containment} this provides a partial answer for the ideal membership question for $\R$. 

\begin{conj}\label{conj:tri-deg-vanishing}
Fix $\nn=(n_1,n_2)\in \Z^2_{\geq1}$ and let $\dd=\one$. Given $\aa=(a_1,a_2)\in \Z^2_{\geq0}$ and $k\in\Z_{\geq0}$ we have that $\dim \overline{S}_{\aa,k}=0$ if and only if $\aa$ and $k$ satisfy one of the following inequalities:
\begin{enumerate}
\item $a_1\geq1$ and  $a_2\geq n_1+1$,
\item $a_2\geq1$ and $a_1\geq n_{2}+1$,
\item $0\leq k\leq a_1a_2-1$, or
\item $k \geq a_1n_{1}+(n_{2}-a_1)a_2+1$.
\end{enumerate}
Moreover, if $k=a_1a_2$ or $k=a_1n_{1}+(n_{2}-a_1)a_2$ then $\dim \overline{S}_{\aa,k}=1$.
\end{conj}

While we are unable to prove the full conjecture, we do prove a large portion of it. In particular, the remaining portion of this section is dedicated to proving Theorem~\ref{thm:tri-deg-vanishing}. This shows that conditions (1), (2), and (3) imply $\dim \overline{S}_{\aa,k}=0$, as well as proves that $\overline{S}_{\aa,a_{1}a_{2}}$ is one dimensional. 

First, using a hypercohomology spectral sequence argument we prove part (1) and (2) of Theorem~\ref{thm:tri-deg-vanishing}. This establishes the sufficiency of conditions (1) and (2) in Conjecture~\ref{conj:tri-deg-vanishing}.

\begin{proof}[Proof of Part (1) and (2) of Theorem~\ref{thm:tri-deg-vanishing}]
Let $\dd=\one$ and consider the Koszul complex of $\O_{\P^{\nn}}$-modules defined on $g_{0},g_{1},\ldots,g_{|\nn|}$:
\begin{equation*}
\fF_{\doot}\coloneqq \left[
\begin{tikzcd}
0 & \lar \O_{\P^{\nn}} & \lar \O_{\P^{\nn}}^{\oplus(|\nn|+1)}(-\one) & \lar \cdots & & \cdots & \lar  \O_{\P^{\nn}}(-(|\nn|+1)\cdot\one) & \lar 0
\end{tikzcd}\right].
\end{equation*}
More precisely $\fF_{\doot}$ is the Koszul complex of $\O_{\P^{\nn}}$-modules where $\fF_{i}=\O_{\P^{\nn}}(-i\cdot\one)^{\oplus\binom{|\nn|+1}{i}}$. Notice that this complex is quasi-isomorphic to zero. Given $\aa=(a_1,a_2)\in\Z^2$ we write $\fF(\aa)_{\doot}$ for the complex $\fF_{\doot}\otimes \O_{\P^{\nn}}(\aa)$. Consider the hypercohomology spectral sequence associated to the complex $\fF(\aa)_{\doot}$, and the global sections functor $\Gamma\left(-,\O_{\P^{\nn}}\right)$, which is defined by:
\[
E^{1}_{p,q}=R^{q}\Gamma\left(\fF(\aa)_{p}\right)=H^{q}\left(\P^{\nn},\fF(\aa)_{p}\right).
\]
This spectral sequence abuts to $\mathbb{H}^{p-q}\left(\fF(\aa)_{\doot}\right)$, which since $\fF_{\doot}$ is quasi-isomorphic to zero is zero. The $E^1$ page of this spectral sequence looks like:

\begin{tikzpicture}
\matrix (m) [matrix of math nodes,
             nodes in empty cells,
             nodes={minimum width=5ex, minimum height=6ex,
                    text depth=1ex,
                    inner sep=0pt, outer sep=0pt,
                    anchor=base},
             column sep=4ex, row sep=4ex]%
{
{\scriptstyle |\nn|} & \scriptstyle H^{|\nn|}\left(\P^{\nn},\fF(\aa)_{0}\right)  & \scriptstyle H^{|\nn|}\left(\P^{\nn},\fF(\aa)_{1}\right)  &  & \cdots & & \scriptstyle H^{|\nn|}\left(\P^{\nn},\fF(\aa)_{\nn}\right)   & \scriptstyle H^{|\nn|}\left(\P^{\nn},\fF(\aa)_{|\nn|+1}\right) & \\
\scriptstyle |\nn|-1 & \scriptstyle H^{|\nn|-1}\left(\P^{\nn},\fF(\aa)_{0}\right)  & \scriptstyle H^{|\nn|-1}\left(\P^{\nn},\fF(\aa)_{1}\right)   & & \cdots & & \scriptstyle H^{|\nn|-1}\left(\P^{\nn},\fF(\aa)_{|\nn|}\right)   & \scriptstyle H^{|\nn|-1}\left(\P^{\nn},\fF(\aa)_{|\nn|+1}\right)  &\\
    [7ex,between origins]
\vdots  & \vdots &  \vdots   &    & &\ddots & \vdots       & \cdots &\\
    [7ex,between origins]
\scriptstyle   1   &    \scriptstyle H^1\left(\P^{\nn},\fF(\aa)_{0}\right)     &  \scriptstyle H^1\left(\P^{\nn},\fF(\aa)_{1}\right)   & & \cdots & &  \scriptstyle H^1\left(\P^{\nn},\fF(\aa)_{|\nn|}\right) & \scriptstyle H^{1}\left(\P^{\nn},\fF(\aa)_{|\nn|+1}\right) & \\
    [5ex,between origins]
\scriptstyle   0   &  \scriptstyle H^0\left(\P^{\nn},\fF(\aa)_{0}\right)       &   \scriptstyle H^0\left(\P^{\nn},\fF(\aa)_{1}\right)             &     &        &  &  \scriptstyle H^0\left(\P^{\nn},\fF(\aa)_{|\nn|}\right)           &  \scriptstyle H^0\left(\P^{\nn},\fF(\aa)_{|\nn|+1}\right) & \\
    [3ex,between origins]
        &  \scriptstyle 0     &  \scriptstyle 1          &           & \cdots & &  \scriptstyle {|\nn|}           & \scriptstyle {|\nn|+1} &\\
};
\draw[shorten <= .1cm, shorten >= .1cm,-stealth] (m-1-3) -- (m-1-3 -| m-1-2.east);
\draw[shorten <= .1cm, shorten >= .1cm,-stealth] (m-2-3) -- (m-2-3 -| m-2-2.east);
\draw[shorten <= .1cm, shorten >= .1cm,-stealth] (m-4-3) -- (m-4-3 -| m-4-2.east);
\draw[shorten <= .1cm, shorten >= .1cm,-stealth] (m-5-3) -- (m-5-3 -| m-5-2.east);

\draw[shorten <= .1cm, shorten >= .1cm,-stealth] (m-1-7) -- (m-1-7 -| m-1-6.east);
\draw[shorten <= .1cm, shorten >= .1cm,-stealth] (m-2-7) -- (m-2-7 -| m-2-6.east);
\draw[shorten <= .1cm, shorten >= .1cm,-stealth] (m-4-7) -- (m-4-7 -| m-4-6.east);
\draw[shorten <= .1cm, shorten >= .1cm,-stealth] (m-5-7) -- (m-5-7 -| m-5-6.east);

\draw[shorten <= .1cm, shorten >= .1cm,-stealth] (m-1-8) -- (m-1-8 -| m-1-7.east);
\draw[shorten <= .1cm, shorten >= .1cm,-stealth] (m-2-8) -- (m-2-8 -| m-2-7.east);
\draw[shorten <= .1cm, shorten >= .1cm,-stealth] (m-4-8) -- (m-4-8 -| m-4-7.east);
\draw[shorten <= .1cm, shorten >= .1cm,-stealth] (m-5-8) -- (m-5-8 -| m-5-7.east);

\draw[thick] (m-1-1.north east) -- (m-6-1.east) ;
\draw[thick] (m-6-1.north) -- (m-6-9.north);
\end{tikzpicture}
By the K\"{u}nneth formula \cite[\href{https://stacks.math.columbia.edu/tag/0BEC}{Tag 0BEC}]{stacks-project} the only possible $q$ for which $H^q\left(\P^{\nn},\O_{\P^{\nn}}(a_1,a_2)\right)$ is non-zero is $q=0,n_1,n_2,$ and $|\nn|$. More specifically:
\[
H^q\left(\P^{\nn},\O_{\P^{\nn}}(a_1,a_2)\right)\cong
\begin{cases}
\K^{\binom{n_1+a_1}{a_1}}\otimes\K^{\binom{n_2+a_2}{a_2}} &\mbox{if } q=0 \mbox{  and  }  a_1\geq0,a_2\geq0\\
\K^{\binom{-a_1-1}{-n_1-a_1-1}}\otimes \K^{\binom{n_2+a_2}{a_2}} &\mbox{if } q=n_1 \mbox{  and  }  a_1\leq-n_1-1,a_2\geq0\\
\K^{\binom{n_1+a_1}{a_1}}\otimes\K^{\binom{-a_2-1}{-n_2-a_2-1}} &\mbox{if } q=n_2 \mbox{  and  }  a_1\geq0,a_2\leq -n_2-1 \\
\K^{\binom{-a_1-1}{-n_1-a_1-1}}\otimes\K^{\binom{-a_2-1}{-n_2-a_2-1}} &\mbox{if } q=n_1+n_1 \mbox{  and  }  a_1\leq -n_1-1,a_2\leq -n_2-1 \\
\end{cases}.
\]
Using this we see that the only non-zero entries on the $E^1$ page occur in rows $q=0,n_1,n_2,$ and $|\nn|$. In fact, the only spots $(p,q)$ on the first page of this spectral sequence, which are non-zero are:
\begin{enumerate}
\item $(p,0)$ for $p$ in the range $p\leq \min\{a_1,a-2,|\nn|+1\}$,
\item $(p,n_1)$ for $p$ in the range $a_1+n_1+1\leq p \leq \min\{a_2,|\nn|+1\}$,
\item $(p,n_2)$ for $p$ in the range $a_2+n_2+1\leq p \leq \min\{a_1,|\nn|+1\}$, and
\item $(p,|\nn|)$ for $p$ in the range $\max\{a_1+n_1+1,a_2+n_2+1\}\leq p\leq |\nn|+1$.
\end{enumerate}
Now consider the $E^2$ page of the spectral sequence. Since $\fF_{\doot}(\aa)$ is a Koszul complex twisted by $\O_{\P^{\nn}}(\aa)$ on this page the 0th row is nothing but the degree $(a_1,a_2)$ strand on the Koszul complex. Moreover, since the cokernel of the Koszul complex is $\overline{S}$ we have that $E^2_{0,0}\cong \overline{S}_{\aa}$. 

On the $j$th page of this spectral sequence the map to $E^j_{0,0}$ has source $E^j_{j,j-1}$. By considering the $E^1$ page this means the only maps to the $E^j_{0,0}$ that may be non-trivial occur when $j=1,n_1+1,n_2+1,$ and $|\nn|+1$. We have already described the map on the $E^1$ page, and so consider the remaining cases.
\begin{itemize}
\item \underline{Page $n_1+1$:} On the $E^{n_1+1}$ page the map to $E^{n_1+1}_{0,0}$ has source $E^{n_1+1}_{n_1+1,n_1}$. Thus, for this map to be trivial it suffices for $E^{1}_{n_1+1,n_1}=0$. By our description of the $E^1$ page above this is true if and only if $n_1+1<a_1+n_1+1$. So if $a_1\geq1$ then $E^{1}_{n_1+1,n_1}=0$, and so the map to $E^{n_1+1}_{0,0}$ is zero.
\item \underline{Page $n_2+1$:} On the $E^{n_2+1}$ page the map to $E^{n_2+1}_{0,0}$ has source $E^{n_2+1}_{n_2+1,n_2}$. Thus, for this map to be trivial it suffices for $E^{1}_{n_2+1,n_2}=0$. By our description of the $E^1$ page above this is true if and only if $n_2+1<a_2+n_2+1$. So if $a_2\geq1$ then $E^{1}_{n_2+1,n_2}=0$, and so the map to $E^{n_2+1}_{0,0}$ is zero.
\item \underline{Page $|\nn|+1$:} On the $E^{|\nn|+1}$ page the map to $E^{|\nn|+1}_{0,0}$ has source $E^{|\nn|+1}_{|\nn|+1,|\nn|}$. Thus, for this map to be trivial it suffices for $E^{1}_{|\nn|+1,|\nn|}=0$. By our description of the $E^1$ page above this is true if and only if $|\nn|+1<\min\{a_1+n_1+1,a_2+n_2+1\}$. So if $a_1\geq n_2+1$ or $a_2\geq n_1+1$ then $E^{1}_{|\nn|+1,|\nn|}=0$, and so the map to $E^{|\nn|+1}_{0,0}$ is zero.
\end{itemize}
Thus, if $a_1,a_2\geq1$ and either $a_1\geq n_2+1$ or $a_2\geq n_1+1$ there are no non-zero maps to $E^j_{0,0}$ for $j\geq2$. Since this spectral sequence abuts to zero $E^{2}_{0,0}\cong \overline{S}_{\aa}\cong 0$.
\end{proof}

We now shift to showing that condition (3) of Conjecture~\ref{conj:tri-deg-vanishing} implies the stated vanishing. Before proving this in full we first consider the special case when $a_1=1$. This will be useful as a base case for our inductive proof of Proposition~\ref{prop:conjecture-1}.

\begin{lemma}\label{lem:tri-vanishing-special-case}
Fix $\nn=(n_1,n_2)\in \Z^2_{\geq1}$ and let $\dd=\one$. Given $\aa=(1,a_2)\in \Z^{2}_{\geq1}$ if $k\leq a_2-1$ then $\dim \overline{S}_{\aa,k}=0$. 
\end{lemma}

\begin{proof}
It is enough to show that if $m\in S_{\aa,k}$ is a monomial then $m\in \R(\nn,\one)$. Now any monomial $m\in S_{\aa,k}$ is of the form $m=x_{\ell}m'$ where $m'$ is a monomial  of bi-degree $(0,a_2)$ and index weighted degree $k-\ell$ supported on $y_0,y_1,\ldots,y_{n_2}$. Since the index weighed degree of $m$ is less than or equal to $a_2-1$ the index weighted degree of $m'$ is less than $a_2-(\ell+1)$. Thus, we may write $m'$ as $m'=y_{0}^{\ell+1}m''$ where $m''$ is a monomial supported on $y_0,y_1,\ldots,y_{n_2}$. So it is enough to show that $x_{\ell}y^{\ell+1}_{0}\in \R(\nn,\one)$.

Towards this we show that $x_{\ell}y^{\ell+1}_{0}\in \R(\nn,\one)$ for $\ell=0,1,\ldots,n_1$ by induction on $\ell$. The base case when $\ell=0$ is clear since $x_{\ell}y_{0}^{\ell+1}=g_{0}$. Now suppose the $x_{\ell'}y^{\ell'+1}_{0}\in \R(\one)$ for all $0\leq \ell'<\ell$. We may write $x_{\ell}y_{0}^{\ell+1}$ as 
\[
x_{\ell}y_{0}^{\ell+1}=x_{\ell}y_{0}y_{0}^{\ell}=\left(g_{\ell}-\sum_{i=1}^{\ell}x_{\ell-i}y_{i}\right)y_{0}^{\ell},
\]
and so it is enough to show that $\left(\sum_{i=1}^{\ell}x_{\ell-i}y_{i}\right)y_{0}^{\ell}$ is in $\R(\one)$. However, each term in this sum is of the form $x_{\ell-i}y_{0}^{\ell-i}y_iy_{0}^{i}$, which by the inductive hypothesis is contained in $\R(\nn,\one)$. 
\end{proof}

\begin{prop}\label{prop:conjecture-1}
Fix $\nn \in \Z^2_{\geq1}$ and let $\dd=\one$. Given $\aa=(a_1,a_2)\in \Z_{\geq1}^2$ if $0\leq k\leq a_1a_2-1$ then $\dim\overline{S}_{\aa,k} =0$. 
\end{prop}

\begin{proof}
Fix $\nn=(n_1,n_2)\in \Z^2_{\geq1}$, $\aa=(a_1,a_2)\in\Z^2$, and $k\in \Z$ such that $k\leq a_1a_2-1$. Without without loss of generality we assume that $n_2\leq n_1$, and that $a_2\leq n_1$ since if $a_{2}>n_1$ then $\dim \overline{S}_{\aa,k}=0$ by parts (1) and (2) of Theorem~\ref{thm:tri-deg-vanishing}. We now proceed by induction on $n_1+a_1$. Note that the base case when $a_1=n_1=1$ follows immediately from Lemma~\ref{lem:tri-vanishing-special-case}. Our inductive hypothesis can now be stated as follows: Let $S'=\K[x_{0},x_{1},\ldots,x_{n'_1},y_{0},y_{1},\ldots,y_{n_2}]$, which we consider with the natural bi-grading, and write $\overline{S}'$ for $S'/\R(n'_1,n_2,\one)$. If $(a'_1,a_2)\in \Z^2_{\geq1}$ and $n'_{1}+a'_{1}<n_1+a_1$ then $\dim \overline{S}'_{(a'_1,a_2),k}=0$ for all $0\leq k\leq a'_{1}a_{2}-1$. 

Now since $\R(\nn,\one)$ is homogeneous with respects to the bi-grading and index weighted grading, Lemma~\ref{lem:homogeneous-wrt-idex}, after shifting accordingly we may consider the short exact sequence 
\begin{center}
\begin{tikzcd}[column sep = 3.5em]
0 \rar & \overline{S}\rar{\cdot x_{n_1}} & \overline{S}\rar & \overline{S}/\langle x_{n_1}\rangle \rar & 0
\end{tikzcd}
\end{center}
as a short exact sequence of graded modules with respect to either the bi-grading or index weighted grading. This gives the following:
\[
\dim \overline{S}_{\aa,k}=\dim \overline{S}_{(a_1-1,a_2),k-n_1}  + \dim \overline{S}/\langle x_{n_1}\rangle_{\aa,k}.
\]
So for the inductive step it is enough for both $\dim \overline{S}_{(a_1-1,a_2),k-n_1}$ and $\dim \overline{S}/\langle x_{n_1}\rangle_{\aa,k}$ to equal zero. 

First we show that $\dim \overline{S}_{(a_1-1,a_2),k-n_1}$ is equal to zero. Notice that by applying the inductive hypothesis in the case when $n'_1=n_1$ and $a'_1=a_1-1$, it is enough for $k-n_1\leq (a_1-1)a_2-1=a_1a_2-a_2-1$. This inequality is true since by our initial assumptions $k\leq a_1a_2-1$ and $a_2\leq n_1$. 

Now we show that $\dim \overline{S}/\langle x_{n_1}\rangle_{\aa,k}$ is equal to zero. By Lemma~\ref{lem:induction-ideal} there is an isomorphism between $\overline{S}/\langle x_{n_1}\rangle$ and $\overline{S}'$ when $n'_1=n_1-1$. In particular, the dimension of $\overline{S}/\langle x_{n_1}\rangle_{\aa,k}$ is equal to the dimension of $\overline{S}'_{\aa,k}$. Applying the inductive hypothesis when $n'_1=n_1-1$ and $a'_1=a_1$, we conclude that $\dim \overline{S}/\langle x_{n_1}\rangle_{\aa,k}=\dim \overline{S}_{\aa,k}=0$.
\end{proof}

We end this section by proving the last remark of Theorem~\ref{thm:tri-deg-vanishing}. Before proving the full claim we first consider the special case when $a_{1}=n_2$ and $a_{2}=n_1$.

\begin{prop}\label{prop:tri-deg-special-case}
Fix $\nn=(n_1,n_2)\in \Z^2_{\geq0}$, and let $\dd=\one$. The dimension of $\overline{S}_{n_2,n_1,n_1n_2}$ is one.  
\end{prop}

\begin{proof}
We use a hypercohomology spectral sequence argument similar to the one used in the proof of parts (1) and (2) of Theorem~\ref{thm:tri-deg-vanishing}. In particular, let $\dd=\one$ and consider the Koszul complex $\fF_{\doot}$ of $\O_{\P^{\nn}}$-modules defined on $g_{0},g_{1},\ldots,g_{|\nn|}$. Notice this complex is quasi-isomorphic to zero.

Writing $\fF(\aa)_{\doot}$ for the complex $\fF_{\doot}\otimes \O_{\P^{\nn}}(\aa)$, we consider the hypercohomology spectral sequence associated to the complex $\fF(\aa)_{\doot}$ and the global sections functor $\Gamma\left(-,\O_{\P^{\nn}}\right)$, which is defined by:
\[
E^{1}_{p,q}=R^{q}\Gamma\left(\fF(\aa)_{p}\right)=H^{q}\left(\P^{\nn},\fF(\aa)_{p}\right).
\]
This spectral sequence abuts to $\mathbb{H}^{p-q}\left(\fF(\aa)\right)$, which is zero since $\fF_{\doot}$ is quasi-isomorphic to zero. Using the fact that $\fF_{i}=\O_{\P^{\nn}}(-i\cdot\one)^{\oplus\binom{|\nn|+1}{i}}$ if $a_{1}=n_2$ and $a_2=n_1$ the $E^1$ page of this spectral sequence looks like:

\begin{tikzpicture}
\matrix (m) [matrix of math nodes,
             nodes in empty cells,
             nodes={minimum width=5ex, minimum height=6ex,
                    text depth=1ex,
                    inner sep=0pt, outer sep=0pt,
                    anchor=base},
             column sep=4ex, row sep=4ex]%
{
{\scriptstyle |\nn|} & \scriptstyle 0 & \scriptstyle 0 &  & \cdots & & \scriptstyle 0   & \scriptstyle \K& \\
\scriptstyle |\nn|-1 & \scriptstyle 0  & \scriptstyle 0  & & \cdots & & \scriptstyle 0   & \scriptstyle 0 &\\
    [7ex,between origins]
\vdots  & \vdots &  \vdots   &    & &\ddots & \vdots       & \cdots &\\
    [7ex,between origins]
\scriptstyle   1   &    \scriptstyle 0     &  \scriptstyle 0   & & \cdots & &  \scriptstyle 0 & \scriptstyle 0 & \\
    [5ex,between origins]
\scriptstyle   0   &  \scriptstyle H^0\left(\P^{\nn},\fF(\aa)_{0}\right)       &   \scriptstyle H^0\left(\P^{\nn},\fF(\aa)_{1}\right)             &     &        &  &  \scriptstyle H^0\left(\P^{\nn},\fF(\aa)_{|\nn|}\right)           &  \scriptstyle H^0\left(\P^{\nn},\fF(\aa)_{|\nn|+1}\right) & \\
    [3ex,between origins]
        &  \scriptstyle 0     &  \scriptstyle 1          &           & \cdots & &  \scriptstyle {|\nn|}           & \scriptstyle {|\nn|+1} &\\
};
\draw[shorten <= .1cm, shorten >= .1cm,-stealth] (m-1-3) -- (m-1-3 -| m-1-2.east);
\draw[shorten <= .1cm, shorten >= .1cm,-stealth] (m-2-3) -- (m-2-3 -| m-2-2.east);
\draw[shorten <= .1cm, shorten >= .1cm,-stealth] (m-4-3) -- (m-4-3 -| m-4-2.east);
\draw[shorten <= .1cm, shorten >= .1cm,-stealth] (m-5-3) -- (m-5-3 -| m-5-2.east);

\draw[shorten <= .1cm, shorten >= .1cm,-stealth] (m-1-7) -- (m-1-7 -| m-1-6.east);
\draw[shorten <= .1cm, shorten >= .1cm,-stealth] (m-2-7) -- (m-2-7 -| m-2-6.east);
\draw[shorten <= .1cm, shorten >= .1cm,-stealth] (m-4-7) -- (m-4-7 -| m-4-6.east);
\draw[shorten <= .1cm, shorten >= .1cm,-stealth] (m-5-7) -- (m-5-7 -| m-5-6.east);

\draw[shorten <= .1cm, shorten >= .1cm,-stealth] (m-1-8) -- (m-1-8 -| m-1-7.east);
\draw[shorten <= .1cm, shorten >= .1cm,-stealth] (m-2-8) -- (m-2-8 -| m-2-7.east);
\draw[shorten <= .1cm, shorten >= .1cm,-stealth] (m-4-8) -- (m-4-8 -| m-4-7.east);
\draw[shorten <= .1cm, shorten >= .1cm,-stealth] (m-5-8) -- (m-5-8 -| m-5-7.east);

\draw[thick] (m-1-1.north east) -- (m-6-1.east) ;
\draw[thick] (m-6-1.north) -- (m-6-9.north);
\end{tikzpicture}

Since $\fF_{\doot}(n_2,n_1)$ is the Koszul complex resolving $\O_{Y}$ twisted by $\O_{\P^{\nn}}(n_2,n_1)$ notice that the $0$th row on this page is nothing by the degree $(n_2,n_1)$-strand of the Koszul complex resolving $\R(\one)$. In particular, since the co-kernel of the Koszul complex resolving $\R(\one)$ is $\overline{S}$ we know that $E^2_{0,0}$ is isomorphic to $\overline{S}_{(n_2,n_1)}$. 

A cohomology computation, similar to the one done in the proof of parts (1) and (2) of Theorem~\ref{thm:tri-deg-vanishing}, shows that every map to $E_{0,0}^{j}$ has trivial source except when $j=1$ and $j=|\nn|+1$. Thus, $E_{0,0}^{|\nn|+1}\cong E_{0,0}^{2} \overline{S}_{(n_2,n_1)}$. Likewise since there are no non-trivial maps to $E_{|\nn|+1,|\nn|}^{j}$ we know that $E_{|\nn|+1,|\nn|}^{|\nn|+1}\cong E_{|\nn|+1,|\nn|}^{1}$ and is is isomorphic to $H^{|\nn|+1}(\P^{\nn},\O_{\P^{\nn}}(-n_1-1,-n_2-1))$. 

The above spectral sequence comes from the bi-complex given by tensoring the \v{C}ech complex on $\P^{\nn}$ with the Koszul complex on $g_{0},g_{1},\ldots,g_{|\nn|}$. Since both of these complexes are homogeneous with respect to the index weighted grading the resulting bi-complex, and hence the associated spectral sequence are also homogeneous with respect to the index weighted grading. Thus all vector spaces and differential appearing in the various pages of the spectral sequence will be graded with respect to the index degree. We would like to determine the index degree of 1-dimensional vector space $H^{|\nn|+1}(\P^{\nn},\O_{\P^{\nn}}(-n_1-1,-n_2-1)$. We compute this as follows: the last term of the Koszul complex is a rank 1-module generated in index weighted degree $\sum_{i} \indeg(g_i)$. The generator of the 1-dimensional vector space $H^{|\nn|+1}(\P^{\nn},\O_{\P^{\nn}}(-n_1-1,-n_2-1)$ corresponds to the Laurent monomial $\frac{1}{x_0x_1\cdots x_{n_1}y_0y_1\cdots y_{n_2}}$. So this vector space has index weighted degree:
\[
\sum_{i=0}^{|\nn|}\indeg(g_i)-\indeg \left(\prod_{i=0}^{n_1}\frac{1}{x_i}\prod_{j=1}^{n_2}\frac{1}{y_j}\right) = \sum_{i=0}^{|\nn|}i-\sum_{j=0}^{n_1}j - \sum_{k=0}^{n_2}k= \binom{n_1+n_2}{2}-\binom{n_1}{2}-\binom{n_2}{2}=n_1n_2.
\]

On the $(|\nn|+1)$-th there is a non-trivial map from $E_{|\nn|+1,|\nn|}^{|\nn|+1}$ to $E_{0,0}^{|\nn|+1}$. Moreover, since this spectral sequence abuts to zero this map must be an isomorphism. Using that $E_{|\nn|+1,|\nn|}^{|\nn|+1}\cong H^{|\nn|+1}(\P^{\nn},\O_{\P^{\nn}}(-n_1-1,-n_2-1))$ and $E_{0,0}^{|\nn|+1}\cong \overline{S}_{(n_2,n_1)}$. As $H^{|\nn|+1}(\P^{\nn},\O_{\P^{\nn}}(-n_1-1,-n_2-1))\cong \K$ this shows that $\dim \overline{S}_{n_2,n_1}=1$. Since this isomorphism respects the index weighted grading it follows that the 1-dimensional vector space of bi-degree $S_{(n_2,n_1)}$ will be supported entirely in index weighted degree $n_1n_2$.
\end{proof}

Finally, we complete the proof of Theorem~\ref{thm:tri-deg-vanishing} by proving the last claim that $\dim \overline{S}_{\aa,a_1a_2}=1$.

\begin{proof}[Proof of Theorem~\ref{thm:tri-deg-vanishing}]
We proceed by induction upon $n_1+n_2$. For the base case note that when $n_1=n_2=1$ the claim is clear as the only case is when $a_1=a_2=1$, which follows from Proposition~\ref{prop:tri-deg-special-case}. Now suppose that $a_1\leq n_2$ and $a_2\leq n_1$. Since we have already prove the claim when $a_1=n_2$ and $a_2=n_1$ without without loss of generality we may suppose that $a_2<n_1$. Considering the short exact sequence 
\begin{center}
\begin{tikzcd}[column sep = 3.5em]
0 \rar & \overline{S}\rar{\cdot x_{n_1}} & \overline{S} \rar & \overline{S}/\langle x_{n_1}\rangle \rar & 0
\end{tikzcd}
\end{center}
we see that:
\[
\dim \overline{S}_{\aa,a_1a_2}=\dim \overline{S}_{(a_1-1,a_2),a_1a_2-n_1}+\dim \overline{S}/\langle x_{n_1}\rangle_{\aa,a_1a_2}.
\]
As $a_2\leq n_1-1$ we see that $a_1a_2-n_1\leq (a_1-1)a_2-1$. So by Proposition~\ref{prop:conjecture-1}  $\dim \overline{S}_{(a_1-1,a_2),a_1a_2-n_1}=0$, and thus, is enough for $\dim \overline{S}/\langle x_{n_1}\rangle_{\aa,a_1a_2}=1$. Letting $S'=\K[x_0,x_{1},\ldots,x_{n_1-1},y_{0},y_{1},\ldots,y_{n_2}]$ with the natural bi-grading by Lemma~\ref{lem:induction-ideal} there is an isomorphism between $\overline{S}/\langle x_{n_1}\rangle$ and $S'/\R((n_1-1,n_2),\dd)$. In particular, $\dim \overline{S}/\langle x_{n_1}\rangle_{\aa,a_1a_2}$ is equal to $\dim S'/\R((n_1-1,n_2),\dd)_{\aa,a_1a_1}$, which by the inductive hypothesis is equal to one. 
\end{proof}

%%%%%%%%%%%%%%%%%%%%%%%%%%%%%%%%%%%%%%%%%%%%%%%%%%%%%%%%%%%%%%%%%%%%%%%%%%%%%%%%%%%%%%%%%%%%%%%%%%%%%%%%%%%%%%%%
\section{non-Vanishing via Generalized Monomial Methods}\label{sec:monomial-techniques}
%%%%%%%%%%%%%%%%%%%%%%%%%%%%%%%%%%%%%%%%%%%%%%%%%%%%%%%%%%%%%%%%%%%%%%%%%%%%%%%%%%%%%%%%%%%%%%%%%%%%%%%%%%%%%%%%

Our main result shows how to construct non-zero syzygies on $\P^{\nn}$ from monomials in $\overline{S}$. This generalizes Lemma~2.3 and Corollary~2.4 in \cite{einErmanLazarsfeld16} to a product of projective spaces. 

A key observation for these generalizations is that the condition of one monomial dividing another monomial used in \cite{einErmanLazarsfeld16} can be weakened to a condition on the index weighted degree. This turns out to be crucial as the notion of when two monomials divide each other in $\overline{S}$ is quite subtle since our regular sequence is not generated by monomials. Before stating and proving this result we first establish a few important definitions and background results. 

\begin{defn}
An element $f\in \overline{S}$ is a monomial of degree $\dd$ if and only if there exists a monomial $m\in S_{\dd}$ such that $\overline{m}=f$ where $\overline{m}$ is the image of $m$ in $\overline{S}$.
\end{defn}

\begin{defn}
An element $\zeta\otimes f\in \Alt^s \overline{S}\otimes \overline{S}$ is a monomial if and only if $\zeta= n_1\wedge\cdots\wedge n_s$ where each $n_i\in \overline{S}$ and $f\in \overline{S}$ are monomials.
\end{defn}

\begin{defn}
Given a finite subset $P\subset \overline{S}$ we write $\det P$ for the wedge product of all elements in $P$, and say $\zeta\in \Alt^{s}P$ if $\zeta=f_1\wedge \cdots\wedge f_s$ where $f_i\in P$ for all $i$. 
\end{defn}

\begin{lemma}\label{lem:stupid}
Let $\phi:V\rightarrow W$ be a map of finite dimensional $\K$-vector spaces, $\{v_1,v_2,\ldots,v_n\}$ be a basis for $V$, and $\{w_1,w_2,\ldots,w_m\}$ be a basis for $W$. If there exists $I\subset \{1,2,\ldots,n\}$ such that
\[
w_1=\sum_{i\in I}\phi(v_i)
\]
then there exists an $i\in I$ such that if we express $\phi(v_i)$ in the given basis as
\[
\phi(v_i)=c_1w_1+c_2w_2+\cdots+c_mw_m
\]
where $c_i\in \K$ then $c_1\neq0$.
\end{lemma}

\begin{proof}
Towards a contradiction suppose that
\[
\phi(v_i)=c_{i,1}w_1+c_{i,2}w_2+\cdots+c_{i,m}w_m
\]
and $c_{i,1}=0$ for all $i\in I$. This means that
\[
w_1=\sum_{i\in I}\phi(v_i)=\sum_{i\in I}\left(c_{i,2}w_2+\cdots+c_{i,m}w_m\right)
\]
which contradicts the fact that $\{w_1,w_2\ldots,w_m\}$ is a basis for $W$.
\end{proof}

With this lemma and these definitions in hand, we can now state the key proposition of this section. 

\begin{prop}\label{prop:nonvanishing-bounds}
Fix $\nn=(n_1,n_2)\in \Z^2_{\geq1}$, $\dd=(d_1,d_2)\in \Z^2_{\geq1}$, $\bb=(b_1,b_2)\in\Z^2$, and $0\leq q \leq |\nn|$. Let $f\in \overline{S}_{q\dd+\bb}$ be a non-zero monomial, and let
\begin{align*}
L(f)&\coloneqq \left\{\begin{matrix} m \\ \text{a monomial of}\\ \text{bi-degree $\dd$}\end{matrix} \; \bigg| \; \begin{matrix} \indeg m \leq \indeg f\end{matrix}\right\}\subset \overline{S}_{\dd}\\
Z(f)&\coloneqq \left\{\begin{matrix} m \\ \text{a monomial of}\\ \text{bi-degree $\dd$}\end{matrix} \; \big| \; mf=0\right\}\subset \overline{S}_{\dd}
\end{align*}
be the set of monomials of bi-degree $\dd$ and of index weighted degree less than $f$ and the set of annihilators of $f$ of bi-degree $\dd$ respectively. Consider the Koszul complex:
\begin{center}
\begin{tikzcd}[column sep = 3em]
\cdots \rar{}& \Alt^{p+1}\overline{S}_{\dd}\otimes \overline{S}_{(q-1)\dd+\bb}\rar{\overline{\partial}_{p+1}}&\Alt^{p}\overline{S}_{\dd}\otimes \overline{S}_{q\dd+\bb}\rar{\overline{\partial}_p}&\Alt^{p-1}\overline{S}_{\dd}\otimes \overline{S}_{(q+1)\dd+\bb}\rar{}&\cdots
\end{tikzcd}
\end{center}
\begin{enumerate}
\item  Given $\zeta \in \Alt^pZ(f)$ the element $\zeta\otimes f\in \ker \overline{\partial}_{p}$.
\item  Given $\zeta\in \bigwedge^{s} \overline{S}_{\dd}$ such that $(\det L(f) \wedge \zeta) \otimes f\neq0$ then $(\det L(f) \wedge \zeta) \otimes f \not\in \img \overline{\partial}_{\#L(f)+s+1}$.
\end{enumerate} 
\end{prop}

\begin{proof}
First let us focus our attention on part (1). By definition, since $\zeta\in \Alt^pZ(f)$ we may write it as $\zeta=\zeta_1\wedge\zeta_2\wedge\cdots\wedge \zeta_p$ where $\zeta_i\in Z(f)\subset \overline{S}_{\dd}$. Thus, we see that
\begin{align*}
\overline{\partial}_{p}\left(\zeta\otimes f\right)=\overline{\partial}_{p}\left(\zeta_1\wedge \zeta_2\wedge\cdots\wedge \zeta_p\otimes f\right)&=\sum_{i=1}^p(-1)^i\zeta_1\wedge \zeta_2\wedge\cdots\wedge \hat{\zeta}_i\wedge \cdots \wedge \zeta_p\otimes (\zeta_if)\\
&=\sum_{i=1}^p(-1)^i\zeta_1\wedge \zeta_2\wedge\cdots\wedge \hat{\zeta}_i\wedge \cdots \wedge \zeta_p\otimes 0=0,
\end{align*}
where the penultimate equality follows from the fact that $\zeta_i\in Z(f)$, and so by definition annihilates $f$.

We now shift to proving part (2). Towards a contradiction suppose that $(\det L(f) \wedge \zeta) \otimes f$ is non-zero and in the image of the map:
\begin{center}
\begin{tikzcd}[ampersand replacement=\&,column sep = 3em]
\Alt^{\#L(f)+s+1}\overline{S}_{\dd}\otimes \overline{S}_{(q-1)\dd+\bb}\rar{\overline{\partial}_{\#L(f)+s+1}}\& \Alt^{\#P(f)+s}\overline{S}_{\dd}\otimes \overline{S}_{q\dd+\bb}
\end{tikzcd}.
\end{center}
This means there exists $\xi_{j}\in \Alt^{\#L(f)+s+1}\overline{S}_{\dd}$ and $g_j\in \overline{S}_{(q-1)\dd+\bb}$ such that:
\begin{equation}\label{eq:mon-method}
\overline{\partial}_{\#L(f)+s+1}\left(\sum_{j=1}^t \xi_j\otimes g_j\right)=\sum_{j=1}^t \overline{\partial}_{\#L(f)+s+1}\left(\xi_j\otimes g_j\right)=(\det L(f) \wedge \zeta) \otimes f.
\end{equation}
By the linearity of $\overline{\partial}_{\#L(f)+s+1}$ we may, without without loss of generality, assume that  $\xi_j\otimes g_j$ is a monomial.

Now the monomials in $\Alt^{\#L(f)+s+1}\overline{S}_{\dd}\otimes \overline{S}_{(q-1)\dd+\bb}$ and in $\Alt^{\#L(f)+s}\overline{S}_{\dd}\otimes \overline{S}_{q\dd+\bb}$ both form spanning sets. Therefore, since $(\det L(f) \wedge \zeta) \otimes f$ is a monomial in $\Alt^{\#L(f)+s}\overline{S}_{\dd}\otimes \overline{S}_{q\dd+\bb}$ and $\xi_j\otimes g_j$ is a monomial in $\Alt^{\#L(f)+s}\overline{S}_{\dd}\otimes \overline{S}_{q\dd+\bb}$ by Lemma~\ref{lem:stupid} for Equation~\eqref{eq:mon-method} to be true it must be the case that $(\det L(f) \wedge \zeta) \otimes f$  appears as a term in $\overline{\partial}_{\#L(f)+s+1}\left(\xi_j\otimes g_j\right)$ for some $j$. 

Since $\xi_j\otimes g_j$ is a monomial we may write $\xi_j\otimes g_j$ as  $\overline{n}_0\wedge \overline{n}_1\wedge \cdots \wedge \overline{n}_{\#P(f)+s}\otimes \overline{g}$ where $g$ and the $n_i$ are monomials in $S$ and $\overline{g}$ and $\overline{n}_i$ are their images in $\overline{S}$. For $(\det L(f) \wedge \zeta) \otimes f$ to appear as a term in $\overline{\partial}_{\#L(f)+s+1}(\overline{n}_0\wedge \overline{n}_1\wedge \cdots \wedge \overline{n}_{\#L(f)+s}\otimes \overline{g})$ without without loss of generality we have that 
\[
 \overline{n}_1\wedge \overline{n}_2\wedge\cdots \wedge \overline{n}_{\#L(f)+s}\otimes (\overline{n}_0\overline{g})=(\det L(f) \wedge \zeta) \otimes f.
\]
This implies two equalities:
\begin{enumerate}
\item $\overline{n}_1\wedge \overline{n}_2\wedge\cdots \wedge \overline{n}_{\#L(f)+s}=(\det L(f) \wedge \zeta)$ as elements in $\Alt^{\#L(f)+s}\overline{S}_{\dd}$, and
\item $\overline{n}_0\overline{g}=f$ as elements in $\overline{S}_{q\dd+\bb}$.
\end{enumerate}
The second of these means that 
\begin{align*}
\indeg(\overline{n}_0)&+\indeg(\overline{g})=\indeg(\overline{n}_0\overline{g})=\indeg(f),
\end{align*}
which implies that $\indeg \overline{n}_0 \leq \indeg f$ meaning $\overline{n}_{0}\in L(f)$. However, combining this fact with the first equality we see that
\[
\overline{n}_0\wedge \overline{n}_1\wedge \cdots \wedge \overline{n}_{\#L(f)+s}\otimes \overline{g}=\overline{n}_0\wedge(\det L(f) \wedge \zeta)\otimes \overline{g}=0
\]
contradicting the fact that $(\det L(f) \wedge \zeta) \otimes f$ is non-zero.
\end{proof}

Immediately from Proposition~\ref{prop:nonvanishing-bounds} we are able to deduce a non-vanishing result giving non-trivial syzygies in a range determined by $L(f)$ and $Z(f)$. 

\begin{cor}\label{cor:monomial-nonvanishing}
Fix $\nn=(n_1,n_2)\in \Z^2_{\geq1}$, $\dd=(d_1,d_2)\in \Z^2_{\geq1}$, and $\bb=(b_1,b_2)\in\Z^2$. Given $0\leq q \leq |\nn|$ and $0\leq k \leq q$ Let $f\in \overline{S}_{q\dd+\bb}$ be a non-zero monomial such that $L(f)\subset Z(f)$ then.
$K_{p,q}^{\overline{R}}\left(\overline{S}(\bb;\dd)\right)$ for all $p$ in the following range:
\[
\#L(f)\leq p \leq \#Z(f).
\]
\end{cor}

\begin{proof}
Fix $\#L(f)\leq p \leq \#Z(f)$ and write $p$ as $p=\#L(f)+s$. Since $s \leq \#\left(Z(f)\setminus L(f)\right)$ we may pick $s$ distinct elements $\zeta_1,\zeta_2,\ldots,\zeta_{s}\in (Z(f)\setminus L(f))$. Set $\zeta = \zeta_1\wedge \zeta_2\wedge\cdots\wedge \zeta_s$. Note that since the $\zeta_i$ are distinct monomials -- and so form part of a basis for $\overline{S}_{\dd}$ -- $\zeta$ is non-zero. By part (1) of Proposition~\ref{prop:nonvanishing-bounds} $(\zeta\wedge \det L(f))\otimes f$ is in the kernel of $\partial_{\#L(f)+s}$, while by part (2) it is not in the image of $\partial_{\#L(f)+s+1}$. Hence it represents a non-zero element in $K_{p,q}^{\overline{R}}\left(\overline{S}(\bb;\dd)\right)$.
\end{proof}

%%%%%%%%%%%%%%%%%%%%%%%%%%%%%%%%%%%%%%%%%%%%%%%%%%%%%%%%%%%%%%%%%%%%%%%%%%%%%%%%%%%%%%%%%%%%%%%%%%%%%%%%%%%%%%
\section{Special Monomails}\label{sec:fqk}
%%%%%%%%%%%%%%%%%%%%%%%%%%%%%%%%%%%%%%%%%%%%%%%%%%%%%%%%%%%%%%%%%%%%%%%%%%%%%%%%%%%%%%%%%%%%%%%%%%%%%%%%%%%%%%

In Section~\ref{sec:monomial-techniques} we showed that given a non-zero monomial $f\in \overline{S}_{q\dd+\bb}$, satisfying certain technical conditions described in Proposition~\ref{prop:nonvanishing-bounds}, one can construct a non-zero syzygy $\zeta_{1}\wedge \zeta_{2}\wedge\cdots \wedge \zeta_{p}\otimes f$ in $K_{p,q}^{\overline{R}}(S(\bb;\dd))$ where $p$ is controlled in part by the annihilators of $f$. We now turn to describing the monomials $f$ we will use in our proof of Theorem~\ref{thm:main}. 

Broadly, the idea is that having fixed $0\leq q\leq |\nn|$ and $\bb\in \Z^2$ we will construct a non-zero monomial $f_{q,k,\bb}\in\overline{S}$ of bi-degree $q\dd$ for every $0\leq k\leq q$. Each $f_{q,k,\bb}$ will play the role of $f$ in Proposition~\ref{prop:nonvanishing-bounds} and Corollary~\ref{cor:monomial-nonvanishing}, and will produce non-trivial syzygies (assuming a few technical conditions) in the range 
\[
\binom{d_1+k}{k}\binom{d_2+(q-k)}{q-k}-(q+2)\leq p \leq r_{\nn,\dd}-\binom{d_1+n_1-k}{n_1-k}\binom{d_2+n_2-(q-k)}{n_2-(q-k)}-(|\nn|+1).
\]

Initially, we will not explicitly define $f_{q,k,\bb}$. Instead we utilize the fact that in certain degrees, described by Theorem~\ref{thm:tri-deg-vanishing}, $(S/\R(\one))_{\aa,t}$ is one dimension. In particular, we define $f_{q,k,\bb}$ in terms of the generator of $(S/\R(\one))_{(k,q-k),k(q-k)}$, which we denote by $\tilde{f}_{q,k}$. Thus, the first part of this section focuses on studying the generator $\tilde{f}_{q,k}$ of $(S/\R(\one))_{(k,q-k),k(q-k)}$. In particular, we show that these $\tilde{f}_{q,k}$ satisfy a series of recursive relations from which it is possible to explicitly write down $\tilde{f}_{q,k}$. 

Following this, in the second part of this section, we define $f_{q,k,\bb}$ and study it's properties. Namely, we  show that $f_{q,k,\bb}$ is well-defined, and then show that it is in fact a non-zero in $\overline{S}$. Moreover, we see that $f_{q,k,\bb}$ is supported on the variables $x_{0},x_{1},\ldots,x_{q-k},y_{0},y_{1},\ldots,y_{k}$.

We then end this section by studying the linear annihilators of $f_{q,k,\bb}$. For example, we show that if $\bb=\zero$ then  $x_{i}f_{q,k,\bb}$ and $y_{j}f_{q,k,\bb}$ equal zero as elements of $\overline{S}$ for $i=0,1,\ldots,(q-k-1)$ and $j=0,1,\ldots,(k-1)$. Understanding these linear annihilators of $f_{q,k,\bb}$ is crucial as it allows us to bound $\#Z(f_{q,k,\bb})$ appearing in Proposition~\ref{prop:nonvanishing-bounds}  and Corollary~\ref{cor:monomial-nonvanishing}.

\subsection{Defining \texorpdfstring{$\tilde{f}_{q,k}$}{$\~{f}_{q,k}$}}
 
As described in Theorem~\ref{thm:tri-deg-vanishing} for certain degrees $(S/\R(\one))_{\aa,t}$ is one dimensional. The goal of this section is to study the generator, unique up to scalar multple, of $(S/\R(\one))_{\aa,t}$. This is useful as the monomial $f_{q,k,\bb}$ we will define in the next subsection is built from these generators. In particular, many of our results about $f_{q,k,\bb}$ given later in this section are built on the understanding of the generator of $(S/\R(\one))_{\aa,t}$. 
 
\begin{defn}
Fix $\nn=(n_1,n_2)\in \Z^2_{\geq1}$. Given $0< q \leq |\nn|$ and $0\leq k \leq q$ such that $q-k\leq n_1$ and $k\leq n_2$ let $\tilde{f}_{q,k}$ be the unique, up to scalar multiplication, non-zero monomial in $(S/\R(\one))_{(k,q-k),k(q-k)}$.
\end{defn}

A crucial property of $\tilde{f}_{q,k}$, which is not immediately obvious from the definition, is that $\tilde{f}_{q,k}$ is supported on $x_{0},x_{1},\ldots,x_{q-k},y_{0},y_{1},\ldots,y_{k}$. In fact, in $\tilde{f}_{q,k}$ is independent, up to the isomorphism described in Lemma~\ref{lem:induction-ideal}, of the $\nn$ so long as $n_1\geq q-k$ and $n_2\geq k$. This is the content for the following lemmas.

\begin{lemma}\label{lem:lem-fqk-independent}
Fix $\nn=(n_1,n_2)\in \Z^2_{\geq1}$, $0< q \leq |\nn|$ and $0\leq k \leq q$ such that $q-k\leq n_1$ and $k\leq n_2$. Let $S'=\K[x_{0},x_{1},\ldots,x_{n_1-i},y_{0},y_{1},\ldots,y_{n_2-j}]$. If $q-k\leq n_1-i$ and $k\leq n_2-j$ and $\tilde{f}'_{q,k}$ is the unique, up to scalar multiplication, non-zero monomial in $(S'/\R(\one))_{(k,q-k),k(q-k)}$ then $\tilde{f}_{q,k}=\tilde{f}'_{q,k}$ under the isomorphism described in Lemma~\ref{lem:induction-ideal}.
\end{lemma}

\begin{proof}
By Lemma~\ref{lem:induction-ideal} there exists an isomorphism of $\K$-vector spaces:
\begin{equation}\label{eq:lem-fqk-independent}
\begin{tikzcd}[column sep = 4em]
\left(\frac{S'}{\R((q-k,k),\one)}\right)_{(k,q-k),k(q-k)} \arrow[r, leftrightarrow, "\sim"] &
\left(\frac{S}{\R(\nn,\one)+\langle x_{n_1-i+1},x_{n_1-i+2},\ldots,x_{n_1},y_{n_2-j+1},y_{n_2-j+2},\ldots,y_{n_2}\rangle}\right)_{(k,q-k),k(q-k)}.
\end{tikzcd}
\end{equation}
The final remark of Theorem~\ref{thm:tri-deg-vanishing} shows that left hand side of \eqref{eq:lem-fqk-independent} is one dimensional, and so right hand side is also one dimension. That said since $\tilde{f}_{q,k}$ is defined to be a representative for the unique, up to scalar multiplication, generator of $S/\R(\nn,\one)_{(k,q-k),k(q-k)}$, and so $\tilde{f}_{q,k}=\tilde{f}'_{q,k}$ under the isomorphism in \eqref{eq:lem-fqk-independent}. 
\end{proof}

\begin{lemma}\label{lem:fqk-relations}
Fix $\nn=(n_1,n_2)\in \Z^2_{\geq1}$, $0< q \leq |\nn|$ and $0\leq k \leq q$ such that $q-k\leq n_1$ and $k\leq n_2$. The following identities hold in $\overline{S}$:
\begin{enumerate}
\item $\tilde{f}_{q,k}=x_{q-k}\tilde{f}_{q-1,k-1}$,
\item $\tilde{f}_{q,k}=y_{k}\tilde{f}_{q-1,k}$, 
\item $\tilde{f}_{q,k}=x_{q-k}y_{k-1}\tilde{f}_{q-2,k-1}$, and
\item $\tilde{f}_{q,k}=x_{q-k-1}y_{k}\tilde{f}_{q-2,k-1}$.
\end{enumerate}
\end{lemma}

\begin{proof}
Parts (3) and (4) follow by combining parts (1) and (2). For part (1) consider a graded component of the maps induced by multiplication by $x_{q-k}$:
\begin{equation}\label{eq:fqk-relations}
\begin{tikzcd}[column sep = 4em]
\left(\frac{S}{\R(\one)}\right)_{(k-1,q-k),(k-1)(q-k)} \rar{x_{q-k}} & \left(\frac{S}{\R(\one)}\right)_{(k,q-k),k(q-k)} 
\end{tikzcd}.
\end{equation}
By Theorem~\ref{thm:tri-deg-vanishing} both the source and target of the map in \eqref{eq:fqk-relations} are one dimensional $\K$-vector spaces. Moreover, $\tilde{f}_{q-1,k-1}$ is a generator for the left hand side and $\tilde{f}_{q,k}$ is a generator for the right hand side. Thus, it is enough to show that $x_{q-k}$ divides $\tilde{f}_{q,k}$. 

Towards a contradiction suppose that $x_{q-k}$ does not divide $\tilde{f}_{q,k}$. Letting $S'=\K[x_{0},x_{1},\ldots,x_{q-k-1},y_{0},y_{1},\ldots,y_{k}]$ by Lemma~\ref{lem:induction-ideal} there is an isomorphism of $\K$-vector spaces:
\begin{equation}\label{eq:fqk-relations-2}
\begin{tikzcd}[column sep = 4em]
\left(\frac{S'}{\R((q-k-1,k),\one)}\right)_{(k,q-k),k(q-k)} \arrow[r, leftrightarrow, "\sim"] &
\left(\frac{S}{\R(\nn,\one)+\langle x_{q-k},x_{q-k+1},\ldots,x_{n_1}\rangle}\right)_{(k,q-k),k(q-k)}.
\end{tikzcd}
\end{equation}
Part (1) of Theorem~\ref{thm:tri-deg-vanishing} implies the source of the map in \eqref{eq:fqk-relations} has dimension zero. However, by Lemma~\ref{lem:lem-fqk-independent} $\tilde{f}_{q,k}$ is non-zero after quotienting $S$ by $\R(\nn,\one)+\langle x_{q-k+1},x_{q-k+1},\ldots,x_{n_1}\rangle$. Thus, since $\tilde{f}_{q,k}$ is not divisible by $x_{q-k}$ it is non-zero in the target of \eqref{eq:fqk-relations}. Hence the target of \eqref{eq:fqk-relations} has dimension one, which is a contradiction. Part (2) of this lemma follows by a similar argument. 
\end{proof}

\begin{remark}
While we will not make use of this, notice that as a consequence of Lemma~\ref{lem:fqk-relations} we are able to write down an explicit representative for $\tilde{f}_{q,k}$. Namely, as $\tilde{f}_{q,0}$ has bi-degree $(0,q)$ and index degree $0$ we know $\tilde{f}_{q,0}=y_{0}^{q}$, and by a similar argument $\tilde{f}_{q,q}=x_{0}^{q}$. So using inductive structure of Lemma~{lem:fqk-relations} with these base cases we can find explicit representatives for $\tilde{f}_{q,k}$. For example,
\[
\tilde{f}_{5,2}=x_{3}y_{1}\tilde{f}_{3,1}=x_{3}y_{1}\left(x_{2}y_{0}\tilde{f}_{1,0}\right)=x_{2}x_{3}y_{0}^{2}y_1.
\]
\end{remark}

\subsection{Defining $f_{q,k,\bb}$}

We are now ready to define $f_{q,k,\bb}$ in terms of $\tilde{f}_{q,k}$. While in Theorem~\ref{thm:main2} we restrict our attention to the case when $\bb\in \Z^{2}_{\geq0}$ we will define $f_{q,k,\bb}$ for a more general range of $\bb$. In particular, under suitable hypothesis we will allow $\bb$ to have negative coordinates. One might hope that these more general $f_{q,k,\bb}$'s may be used to extend Theorem~\ref{thm:main2} to additional cases of $\bb$, even though we do not carry this out here. 

\begin{notation}
Given a monomial $\xx^{\vv}\yy^{\yy}$ in $S$ (or $\overline{S}$) we write $(\xx^{\vv}\yy^{\yy})^{\dd}$ for the monomial $\xx^{d_1\vv}\yy^{d_2\yy}$ in $S$ (or $\overline{S}$).
\end{notation}

\begin{defn}\label{defn:fqkb}
Fix $\nn=(n_1,n_2)\in \Z^2_{\geq1}$, $\dd=(d_1,d_2)\in \Z^2_{\geq1}$, and $\bb\in \Z^2$. Let $0< q \leq |\nn|$ and $0\leq k \leq q$ such that $q-k\leq n_1$ and $k\leq n_2$. If $d_1>q-k+b_1$ and $d_2>k+b_2$ then define $f_{q,k,\bb}$ to be:
\[
f_{q,k,\bb}\coloneqq
\begin{cases}
\left(x_0\cdots x_{q-1}\right)^{d_1-1}x_{q}^{q+b_1}y_{0}^{b_2}\tilde{f}_{q,0}^{\dd}& \mbox{if } k=0 \\
\left(x_0\cdots x_{q-k-1}\right)^{d_1-1}x_{q-k}^{q-k+b_1}\left(y_0\cdots y_{k-1}\right)^{d_2-1}y_{k}^{k+b_2}\tilde{f}_{q,k}^{\dd} & \mbox{if } k\neq0
\end{cases}.
\]
\end{defn}

Note that from the definition it is not necessarily clear that $f_{q,k,\bb}$ is in fact an element of $\overline{S}$. In particular, since $b_1$ and $b_2$ may be negative the terms $x_{q-k}^{q-k+b_1}$ and $y_{k}^{k+b_2}$ appearing in the definition of $f_{q,k,\bb}$ need not be monomials. In fact, if \textit{both} $b_1$ and $b_2$ are sufficiently negative $f_{q,k,\bb}$ is not an element of $\overline{S}$. 

However, the following lemma shows that as long as  at least one of $q-k+b_1$ and $k+b_2$ are non-negative and $\dd\gg\bb$ we need not worry about this except in a few edge cases (i.e. when $k=0$ or $k=q$). The key insight is that the relations in Lemma~\ref{lem:fqk-relations} provide the following alternative definitions for $f_{q,k,\bb}$ when exactly one of $q-k+b_1$ and $k+b_2$ is negative are immediate.

\begin{lemma}\label{lem-fqkb-well-defined}
Fix $\nn=(n_1,n_2)\in \Z^2_{\geq1}$, $\dd=(d_1,d_2)\in \Z^2_{\geq1}$, and $\bb\in \Z^2$. Let $0< q \leq |\nn|$ and $0\leq k \leq q$ such that $q-k\leq n_1$ and $k\leq n_2$.
\begin{enumerate}
\item Suppose $q-k+b_1<0$ and $k+b_2\geq0$. If $d_1>|q-k+b_1|$ and $k\neq0$ then $f_{q,k,\bb}\in \overline{S}$ and 
\[
f_{q,k,\bb}=
\left(x_0\cdots x_{q-k-1}\right)^{d_1-1}x_{q-k}^{q-k+b_1+d_1}\left(y_0\cdots y_{k-1}\right)^{d_2-1}y_{k}^{k+b_2}\tilde{f}_{q-1,k-1}^{\dd}.
\]
\item Suppose $q-k+b_1\geq0$ and $k+b_2<0$. If $d_2>|k+b_2|$ then $f_{q,k,\bb}\in \overline{S}$ and 
\[
f_{q,k,\bb}=
\begin{cases}
\left(x_0\cdots x_{q-1}\right)^{d_1-1}x_{q}^{q+b_1}y_{0}^{b_2+d_2}\tilde{f}_{q-1,0}^{\dd}& \mbox{if } k=0 \\
\left(x_0\cdots x_{q-k-1}\right)^{d_1-1}x_{q-k}^{q-k+b_1}\left(y_0\cdots y_{k-1}\right)^{d_2-1}y_{k}^{k+b_2+d_2}\tilde{f}_{q-1,k}^{\dd} & \mbox{if } k\neq0
\end{cases}.
\]
\end{enumerate}
\end{lemma}

\begin{proof}
Apply Lemma~\ref{lem:fqk-relations}.
\end{proof}

\begin{remark}
Since $\tilde{f}_{q,k}$ has bi-degree $(k,q-k)$ the monomial $f_{q,k,\bb}$ has bi-degree $q\dd$. This is key as we wish to apply these $f_{q,k,\bb}$ as in Proposition~\ref{prop:nonvanishing-bounds}. 
\end{remark}

The remainder of this section, is dedicated to proving that for certain $\bb$ the monomial $f_{q,k,\bb}$ is non-zero as an element of $\overline{S}$, and describing a certain subset of $(\R:f_{q,k,\bb})$. An important property, key to proving both of these, is that $(f_{q,k,\bb}/\remd(f_{q,k,\bb}))^{1/\dd}$ is equal to $\tilde{f}_{q',k'}$ for some $q'$ and $k'$. 

\begin{lemma}\label{lem:floor-of-fqk}
Fix $\nn=(n_1,n_2)\in \Z^2_{\geq1}$, $\dd=(d_1,d_2)\in \Z^2_{\geq1}$, and $\bb\in \Z^{2}$. Further fix $0< q \leq |\nn|$ and $0\leq k \leq q$ such that $q-k\leq n_1$ and $k\leq n_2$. Suppose that $|(q-k)+b_1|<d_1$ and $|k+b_2|<d_2$.
\begin{enumerate}
\item If $0\geq (q-k)+b_1$ and $0\geq k+b_2$ then $\left(\frac{f_{q,k,\bb}}{\remd(f_{q,k,\bb})}\right)^{1/\dd}$ is equal to $\tilde{f}_{q,k}$.
\item If $q-k+b_1<0$, $k+b_2\geq0$, and $k\neq0$ then $\left(\frac{f_{q,k,\bb}}{\remd(f_{q,k,\bb})}\right)^{1/\dd}$ is equal to $\tilde{f}_{q-1,k-1}$. \\
\item If $q-k+b_1\geq0$, $k+b_2<0$ then $\left(\frac{f_{q,k,\bb}}{\remd(f_{q,k,\bb})}\right)^{1/\dd}$ is equal to $\tilde{f}_{q-1,k}$.
\end{enumerate}
\end{lemma}

\begin{proof}
We only prove part (1) as the the remaining parts follow in a similar manner from Lemma~\ref{lem-fqkb-well-defined}. First we handle the case when $k=0$. The key facts are that $\remd(\tilde{f}_{q,0}^{\dd})$ is equal to $\tilde{f}_{q,0}$, and that $\remd(x_q^{q+b_1}) =x_q^{q+b_1}$ and $\remd(y_{0}^{b_1})=y_0^{b_2}$ since $q+b_1<d_{1}$ and $b_2<d_2$ respectively. Computing:
\[
\left(\frac{f_{q,0,\bb}}{\remd(f_{q,0,\bb})}\right)^{1/\dd}=\left(\frac{\left(x_0\cdots x_{q-1}\right)^{d_1-1}x_{q}^{q+b_1}y_{0}^{b_2}\tilde{f}_{q,0}^{\dd}}{\left(x_0\cdots x_{q-1}\right)^{d_1-1}x_{q}^{q+b_1}y_0^{b_2}}\right)^{1/\dd}=\left(\tilde{f}_{q,0}^{\dd}\right)^{1/\dd}=\tilde{f}_{q,0}.
\]
The case when $k\neq0$ is essentially the same. Again the key facts are that $\remd\left(\tilde{f}_{q,k}^{\dd}\right)=\tilde{f}_{q,k}$, and that $\remd\left(x_{q-k}^{q-k+b_1}\right) =x_{q-k}^{q-k+b_1}$ and $\remd\left(y_{k}^{k+b_2}\right)=y_{k}^{k+b_2}$ since $(q-k+b_1)<d_{1}$ and $k+b_2<d_2$. From these then the result follows from the computation:
\[
\left(\frac{f_{q,k,\bb}}{\remd(f_{q,k,\bb})}\right)^{1/\dd}=\left(\frac{\left(x_0\cdots x_{q-k-1}\right)^{d_1-1}x_{q-k}^{q-k+b_1}\left(y_0\cdots y_{k-1}\right)^{d_2-1}y_{k}^{k+b_2}\tilde{f}_{q,k}^{\dd}}{\left(x_0\cdots x_{q-k-1}\right)^{d_1-1}x_{q-k}^{q-k+b_1}\left(y_0\cdots y_{k-1}\right)^{d_2-1}y_{k}^{k+b_2}}\right)^{1/\dd}=\left(\tilde{f}_{q,k}^{\dd}\right)^{1/\dd}=\tilde{f}_{q,k}.
\]
\end{proof}

Note the conditions that $d_1>|q-k+b_1|$ and $d_2>|k+b_2|$ ensure that $(q-k+b_1)$ and $k+b_2$ remain unchanged modulo $d_1$ and $d_2$ respectively. This, and Lemma~\ref{lem-fqkb-well-defined}, are the source of conditions appearing in Theorem~\ref{thm:main}. Using this previous lemma together with Proposition~\ref{prop:containment} we conclude that $f_{q,k,\bb}\neq0$ as an element of $\overline{S}$.

\begin{lemma}\label{lem-fqkb-nonzero}
Fix $\nn=(n_1,n_2)\in \Z^2_{\geq1}$, $\dd=(d_1,d_2)\in \Z^2_{\geq1}$, and $\bb\in \Z^2$. Further fix $0< q \leq |\nn|$ and $0\leq k \leq q$ such that $q-k\leq n_1$ and $k\leq n_2$. Suppose $|q-k+b_1|< d_1$ and $|k+b_2|<d_2$. If one of the following pairs of inequalities hold:
\begin{enumerate}
\item $0\geq (q-k)+b_1$ and $0\geq k+b_2$,
\item $q-k+b_1<0$, $k+b_2\geq0$, or
\item $q-k+b_1\geq0$, $k+b_2<0$
\end{enumerate}
then the monomial $f_{q,k,\bb}\neq0$ as an element of $\overline{S}$.
\end{lemma}

\begin{proof}
By Proposition~\ref{prop:containment} $f_{q,k,\bb}\not\in \R(\dd)$ if and only if $\left(\frac{f_{q,k,\bb}}{\remd(f_{q,k,\bb})}\right)^{1/\dd}\not\in \R(\one)$. By Lemma~\ref{lem:floor-of-fqk} $\left(\frac{f_{q,k,\bb}}{\remd(f_{q,k,\bb})}\right)^{1/\dd}$ is equal to $\tilde{f}_{q',k'}$ for some $q'$ and $k'$. So $f_{q,k,\bb}\not\in \R(\dd)$ if and only if $\tilde{f}_{q',k'}\not\in \R(\one)$. However, by construction $\tilde{f}_{q',k'}\not\in \R(\one)$.
\end{proof}

\subsection{Linear Annihilators of $f_{q,k,\bb}$}

Finally, we show that while $f_{q,k,\bb}$ is non-zero it is annihilated by $x_i$ and $y_j$ for where $i$ and $j$ are by $q,k,$ and $\bb$. For example, if $\bb=\zero$ then  $x_{i}f_{q,k,\bb}$ and $y_{j}f_{q,k,\bb}$ are equal to zero for a number of $i=0,1,\ldots,(q-k-1)$ and $j=0,1,\ldots,(k-1)$. Understanding these linear annihilators of $f_{q,k,\bb}$ is crucial to the proof of Theorem~\ref{thm:main} and Theorem~\ref{thm:main2} as it allows us to bound the $\#Z(f_{q,k,\bb})$, which is crucial to Corollary~\ref{cor:monomial-nonvanishing}.

\begin{prop}\label{prop:linear-annhilators}
Fix $\nn=(n_1,n_2)\in \Z^2_{\geq1}$ and $\dd=(d_1,d_2)\in \Z^2_{\geq1}$. Further fix integers $0<q\leq|\nn|$ and $0\leq k\leq q$ such that $(q-k)\leq n_1$ and $k\leq n_2$. Suppose $|q-k+b_1|< d_1$ and $|k+b_2|<d_2$.
\begin{enumerate} 
\item If $0\geq (q-k)+b_1$ and $0\geq k+b_2$ then 
\[
\langle x_0,x_1,\ldots,x_{q-k-1},y_{0},y_{1},\ldots,y_{k-1}\rangle \subset \left(0:_{\overline{S}}f_{q,k,\bb}\right).
\]
\item If $q-k+b_1<0$, $0k+b_2\geq0$, and $k\neq0$ then 
\[
\langle x_0,x_1,\ldots,x_{q-k-1},y_{0},y_{1},\ldots,y_{k-2}\rangle \subset \left(0:_{\overline{S}}f_{q,k,\bb}\right).
\]
\item If $q-k+b_1\geq0$, $k+b_2<0$ then
\[
\langle x_0,x_1,\ldots,x_{q-k-2},y_{0},y_{1},\ldots,y_{k-1}\rangle \subset \left(0:_{\overline{S}}f_{q,k,\bb}\right).
\]
\end{enumerate}
\end{prop}

\begin{proof}
We begin by proving part (1). First we handle the case when $k=0$. Fixing an integer $0\leq i\leq q-1$, we wish to show that $x_{i}f_{q,k,\bb}\in \R$. By Proposition~\ref{prop:containment} $x_{i}f_{q,k,\bb}\in \R(\dd)$ if and only if $\left(\frac{x_if_{q,k,\bb}}{\remd( x_if_{q,k,\bb})}\right)^{1/\dd}\in \R(\one)$. Using that $0\leq q-k+b_1< d_1$ and $0\leq k+b_2<d_2$ and performing a computation analogous to the one in Lemma~\ref{lem:floor-of-fqk} we find:
\[
\left(\frac{x_if_{q,0,\bb}}{\remd(x_if_{q,0,\bb})}\right)^{1/\dd}=\left(\frac{x_i\left(x_0\cdots x_{q-1}\right)^{d_1-1}x_{q}^{q+b_1}y_0^{b_2}\tilde{f}_{q,0}^{\dd}}{\left(x_0\cdots\widehat{x}_{i}\cdots x_{q-1}\right)^{d_1-1}x_{q}^{q+b_1}y_{0}^{b_2}}\right)^{1/\dd}=\left(x_i^{d_1}\tilde{f}_{q,0}^{\dd}\right)^{1/\dd}=x_i\tilde{f}_{q,0}.
\]
Now $x_i\tilde{f}_{q,0}$ has bi-degree $(1,q)$ and index weighted degree $i$, and so $i\leq q-1$. Theorem~\ref{thm:tri-deg-vanishing} implies that $x_i\tilde{f}_{q,0}\in \R(\one)$. 

Turing to the case when $k>0$ fix an integer $0\leq i\leq q-k-1$. We wish to show that $x_if_{q,k,\bb}\in \R(\dd)$. By Proposition~\ref{prop:containment} $x_{i}f_{q,k,\bb}\in \R(\dd)$ if and only if $\left(\frac{x_if_{q,k,\bb}}{\remd(x_if_{q,k,\bb})}\right)^{1/\dd}\in \R(\one)$. Using that $0\leq q-k+b_1< d_1$ and $0\leq k+b_2<d_2$ and performing a computation analogous to the one in Lemma~\ref{lem:floor-of-fqk} we find:
\begin{align*}
\left(\frac{x_if_{q,k,\bb}}{\remd(x_if_{q,k,\bb})}\right)^{1/\dd}=\left(\frac{x_i\left(x_0\cdots x_{q-k-1}\right)^{d_1-1}x_{q-k}^{q-k+b_1}\left(y_0\cdots y_{k-1}\right)^{d_2-1}y_{k}^{k+b_2}\tilde{f}_{q,k}^{\dd}}{\left(x_0\cdots\widehat{x}_{i}\cdots x_{q-k-1}\right)^{d_1-1}x_{q-k}^{q-k+b_1}\left(y_0\cdots y_{k-1}\right)^{d_2-1}y_{k}^{k+b_2}}\right)^{1/\dd}
=\left(x_i^{d_1}\tilde{f}_{q,k}^{\dd}\right)^{1/\dd}=x_i\tilde{f}_{q,k},
\end{align*}
and so it enough to show that $x_i\tilde{f}_{q,k}\in \R(\one)$. Computing we find that $\deg(x_i\tilde{f}_{q,k})=(k+1,q-k)$ and $\indeg(x_i\tilde{f}_{q,k}))=k(q-k)+i$. Finally, notice that since $0\leq i\leq q-k-1$ we have that:
\[
k(q-k)+i\leq k(q-k)+(q-k-1)=(k+1)(q-k)-1,
\]
and so by Theorem~\ref{thm:tri-deg-vanishing} $\dim \overline{S}_{(k+1,q-k),k(q-k)+i}=0$, which implies that $x_i\tilde{f}_{q,k}\in \R(\one)$.

The argument the $y_j$'s is similar. Fixing a natural number $0\leq j\leq k-1$, we wish to show that $y_jf_{q,k,\bb}\in \R(\dd)$. Again using Proposition~\ref{prop:containment} it is enough to show that $\left(\frac{y_jf_{q,k,\bb}}{\remd(y_jf_{q,k})}\right)^{1/\dd}\in \R(\one)$. A computation analogous to the one in the previous cases shows that: 
\begin{align*}
\left(\frac{y_jf_{q,k,\bb}}{\remd(y_jf_{q,k,\bb})}\right)^{1/\dd}=\left(\frac{y_j\left(x_0\cdots x_{q-k-1}\right)^{d_1-1}x_{q-k}^{q-k+b_1}\left(y_0\cdots y_{k-1}\right)^{d_2-1}y_{k}^{k+b_2}\tilde{f}_{q,k}^{\dd}}{\left(x_0\cdots x_{q-k-1}\right)^{d_1-1}x_{q-k}^{q-k+b_1}\left(y_0\cdots\widehat{y}_{j}\cdots y_{k-1}\right)^{d_2-1}y_{k}^{k+b_2}}\right)^{1/\dd}
=\left(y_j^{d_2}\tilde{f}_{q,k}^{\dd}\right)^{1/\dd}=y_j\tilde{f}_{q,k},
\end{align*}
and so it enough to show that $y_{j}\tilde{f}_{q,k}\in \R(\one)$. The bi-degree of this element is $(k,q-k+1)$ and its index weighted degree is $k(q-k)+)$. Finally, notice that since $0\leq j\leq k-1$ we have that:
\[
k(q-k)+j\leq k(q-k)+(k-1)=k(q-k+1)-1,
\]
and so by Theorem~\ref{thm:tri-deg-vanishing} $\dim \overline{S}_{(k,q-k+1)k(q-k)+j}=0$, which implies that $y_j\tilde{f}_{q,k}\in \R(\one)$. Parts (2) and (3) follow in a similar fashion.
\end{proof}

%%%%%%%%%%%%%%%%%%%%%%%%%%%%%%%%%%%%%%%%%%%%%%%%%%%%%%%%%%%%%%%%%%%%%%%%%%%%%%%%%%%%%%%%%%%%%%%%%%%%%%%%%%%%%%%
\section{The Key Case - $K_{p,q}\left(\P^{q-k}\times \P^{k},\bb;\dd\right)$}\label{sec:special-case}
%%%%%%%%%%%%%%%%%%%%%%%%%%%%%%%%%%%%%%%%%%%%%%%%%%%%%%%%%%%%%%%%%%%%%%%%%%%%%%%%%%%%%%%%%%%%%%%%%%%%%%%%%%%%%%%%

In this section, we prove a special case of Theorem~\ref{thm:main}. Specifically, we fix $0\leq k\leq q$ and consider $K_{p,q}(\P^{q-k}\times \P^{k}, \bb;\dd)$. Following our heuristic that the non-vanishing of $K_{p,q}(\nn,\bb;\dd)$ is controlled by subvarieties of the form $\P^{i}\times\P^{j}$ where $i+j=q$, we see that the case of $\P^{q-k}\times\P^{k}$ may be simpler than the general case, since the only subvariety of this form is $\P^{q-k}\times\P^{k}$ itself. This special case is crucial to our proof of the full theorem. In particular, our proof of the general case uses a series of arguments to reduce Theorem~\ref{thm:main} to the following special case.

\begin{theorem}\label{thm:special}
Fix integers $0\leq k\leq q$, $\dd\in \Z_{>1}^2$, and $\bb\in \Z^2_{\geq0}$. If $0\leq q-k+b_1< d_1$ and $0\leq k+b_2<d_2$ and
\[
\frac{d_1}{d_2}b_2-b_1<q-k+1\quad \quad \text{and} \quad \quad \frac{d_2}{d_1}b_1-b_2<k+1
\]
 then $K_{p,q}((q-k,k),\bb;\dd)\neq0$ for $p = r_{(q-k,k),\dd}-(q+1)$
\end{theorem}

Before proving Theorem~\ref{thm:special} we need two lemmas regarding Hilbert functions. The first lemma shows that the Hilbert function of an ideal $J\subset S$ can be bounded below in terms of the number of the linearly independent forms of total degree one. 

\begin{lemma}\label{lem:hilbert-function-bound}
If $J\subset S$ is a homogeneous ideal and $K\subset  \left\langle J_{(1,0)},J_{(0,1)}\right\rangle$ then
\[
\HF(\dd,J)\geq r_{\nn,\dd}-\binom{d_1+n_1-\dim K_{(1,0)}}{n_1-\dim K_{(1,0)}}\binom{d_2-n_2-\dim K_{(0,1)}}{n_2-\dim K_{(0,1)}}
\]
\end{lemma}

\begin{proof}
Since $K$ is gender by monomials of total degree one, the quotient $S/K$ is exactly a bi-graded polynomial ring with $n_1-\dim K_{(1,0)}$ $x$-variables and $n_2-\dim K_{(0,1)}$ $y$-variables. It follows that $\HF(\dd,S/K)$ equals $\binom{d_1+n_1-\dim K_{(1,0)}}{n_1-\dim K_{(1,0)}}\binom{d_2-n_2-\dim K_{(0,1)}}{n_2-\dim K_{(0,1)}}$, and since $\HF(\dd,S)=\HF(\dd,K)+\HF(\dd,S/K)$, we see that $\HF(\dd,K)$ equals the right hand side of the displayed equation in the lemma. Now, since $K$ is a subset of $J$ we have that $\HF(\dd,K)\leq \HF(\dd,J)$ yielding the desired result. 
\end{proof}

The second lemma concerns the Hilbert function of $\R(\dd)$.

\begin{lemma}\label{lem:hilbert-function-I}
Fix integers $0\leq k\leq q$ and $\dd\in \Z_{>1}^2$ and let $n_1=q-k$ and $n_2=k$:
\[
\HF(\dd,\R(\dd))=(q-k)+k+1=q+1.
\]
\end{lemma}

\begin{proof}
Since $\R(\dd)$ is generated in bi-degree $\dd$ by $g_{0},g_{1},\ldots,g_{\nn}$ it is enough to show that these $g_{\ell}$'s are linearly independent over $\K$. This follow from the fact that the index weighted degree induces a $\Z$-grading on $S$, and that index weighted degree of each $g_{t}$ is distinct, Lemma~\ref{lem:homogeneous-wrt-idex}.
\end{proof}

\begin{proof}[Proof of Theorem~\ref{thm:special}]
As we are considering the case of $\P^{q-k}\times\P^{k}$, i.e. when $\nn=(q-k,k)$, thoughtout this proof we let $S=\K[x_0,x_1,\ldots,x_{q-k},y_0,y_1,\ldots,y_{k}]$. By Proposition~\ref{prop:cohen-macaulay} the inequalities on $\bb$ and $\dd$ in the hypothesis of the theorem ensure that $S(\bb;\dd)$ is Cohen-Macaulay as an $R$-module. Hence by the Artinian reduction explained in Corollary~\ref{cor:artinian-reduction} we know that there exists a natural isomorphism between $K_{p,q}((q-k,k),\bb;\dd)$ and $K_{p,q}^{\overline{R}}\left(\overline{S}(\bb;\dd)\right)$. So it is enough to prove that $K_{p,q}^{\overline{R}}\left(\overline{S}(\bb;\dd)\right)$ is non-zero for $p = r_{(q-k,k),\dd}-(q+1)$. 

We do this by applying Corollary~\ref{cor:monomial-nonvanishing} to the monomial $f_{q,k,\bb}$ described in Definition~\ref{defn:fqkb}. However, before we do this we check that $f_{q,k,\bb}$ satisfies the conditions need for Corollary~\ref{cor:monomial-nonvanishing}. In particular, we check that $f_{q,k,\bb}$ is a well-defined monomial, which is non-zero in $\overline{S}_{q\dd+\bb}$. The fact that $f_{q,k,\bb}$ is a well-defined monomial follows from the fact that both coordinates of $\bb=(b_1,b_2)$ are non-negative (see Lemma~\ref{lem-fqkb-well-defined}). Moreover, since $0\leq q-k+b_1<d_1$ and $0\leq k+b_2<d_2$ we know that $f_{q,k,\bb}$ is non-zero as an element of $\overline{S}$ (see Lemma~\ref{lem-fqkb-nonzero}). 

Thus, by the monomial methods from Corollary~\ref{cor:monomial-nonvanishing} if $L(f_{q,k,\bb})\subset Z(f_{q,k})$ then $K_{p,q}^{\overline{R}}\left(\overline{S}(\bb;\dd)\right)\neq0$ for all $p$ in  
\[
\#L(f_{q,k,\bb})\leq p \leq \#Z(f_{q,k,\bb}).
\] In particular, using the trivial upper bound that $\#L(f_{q,k,\bb})\leq \#Z(f_{q,k,\bb})$ gives non-vanishing for $p=\#Z(f_{q,k})$. Thus, it is enough to i) show that $L(f_{q,k,\bb})\subset Z(f_{q,k,\bb})$ and ii) give a lower bound on $\#Z(f_{q,k,\bb})$ that is also an upper bound on $\#L(f_{q,k,\bb})$. 

Towards part (i) recall that 
\[
Z(f_{q,k,\bb})= \left\{\begin{matrix} m \\ \text{a monomial of}\\ \text{bi-degree $\dd$}\end{matrix} \; \bigg| \; mf_{q,k,\bb}=0\right\}\subset \overline{S}_{\dd},
\] 
and so $Z(f_{q,k,\bb})$ is equal to $(0:_{\overline{S}}f_{q,k,\bb})_{\dd}\subset \overline{S}$. By Proposition~\ref{prop:linear-annhilators} the ideal $\langle x_{0},x_{1},\ldots,x_{q-k-1},y_0,y_{1},\ldots,y_{k-1}\rangle$ is contained in $(\R(\dd):_{S}f_{q,k,\bb})$, and so the degree $\dd$ part of $\langle x_{0},x_{1},\ldots,x_{q-k-1},y_0,y_{1}\ldots,y_{k-1}\rangle\overline{S}$ is contained in $Z(f_{q,k,\bb})=(0:_{\overline{S}}f_{q,k,\bb})_{\dd}$. Thus, it is enough to show:
\[
L(f_{q,k,\bb})= \left\{\begin{matrix} m \\ \text{a monomial of}\\ \text{bi-degree $\dd$}\end{matrix} \; \bigg| \; \begin{matrix}\indeg m\leq \indeg f_{q,k,\bb}\end{matrix}\right\}\subset \left(\langle {x}_{0},x_{1},\ldots,{x}_{q-k-1},{y}_0,y_{1},\ldots,{y}_{k-1}\rangle\overline{S}\right)_{\dd}.
\]
Since $n_1=q-k$ and $n_2=k$ the only monomial of bi-degree $\dd$ not contained in $\langle {x}_{0},x_{1},\ldots,{x}_{q-k-1},{y}_0,y_{1},\ldots,{y}_{k-1}\rangle\overline{S}$ is ${x}_{q-k}^{d_1}{y}_{k}^{d_2}$. However, since $n_1=q-k$ and $n_2=k$ we know that ${x}_{q-k}^{d_1}{y}_{k}^{d_2}=g_{|\nn|}$, and so ${x}_{q-k}^{d_1}{y}_{k}^{d_2}=0$ as an element of $\overline{S}$. This gives the following containments:
\[
L(f_{q,k,\bb})\subset \langle {x}_{0},x_{1},\ldots,{x}_{q-k-1},{y}_0,y_{1},\ldots,{y}_{k-1}\rangle\overline{S}\subset Z(f_{q,k,\bb}).
\]

Shifting our focus to step (ii), and giving a lower bound for $\#Z(f_{q,k,\bb})$, note that:
\[
\#Z(f_{q,k,\bb})=\HF(\dd,(0:_{\overline{S}}f_{q,k,\bb}))=\HF(\dd,(\R:_{S}f_{q,k,\bb}))-\HF(\dd,\R).
\]
Utilizing the fact that $\langle x_{0},x_{1},\ldots,x_{q-k-1},y_0,y_{1},\ldots,y_{k-1}\rangle$ is contained in $(\R:_{S}f_{q,k,\bb})$, Proposition~\ref{prop:linear-annhilators}, together with Lemmas ~\ref{lem:hilbert-function-bound} and ~\ref{lem:hilbert-function-I} we get the desired result
\begin{align*}
\#Z(f_{q,k,\bb})=\HF(\dd,(\R :_{S}f_{q,k,\bb}))-\HF(\dd,\R)\geq r_{q-k,k\dd}+1-\binom{d_1}{0}\binom{d_2}{0}-\HF(\dd,\R)=r_{(q-k,k),\dd}-(q+1).
\end{align*}

\end{proof}
%%%%%%%%%%%%%%%%%%%%%%%%%%%%%%%%%%%%%%%%%%%%%%%%%%%%%%%%%%%%%%%%%%%%%%%%%%%%%%%%%%%%%%%%%%%%%%%%%%%%%%%%%%%%%%%%
\section{Proof of Main Theorems}\label{sec:proof-of-theorem}
%%%%%%%%%%%%%%%%%%%%%%%%%%%%%%%%%%%%%%%%%%%%%%%%%%%%%%%%%%%%%%%%%%%%%%%%%%%%%%%%%%%%%%%%%%%%%%%%%%%%%%%%%%%%%%%%

We are now ready to prove our main results: Theorem~\ref{thm:main}, Corollary~\ref{cor:main}, and Theorem~\ref{thm:main2}. By combining part (1) of Proposition~\ref{prop:nonvanishing-bounds} and Proposition~\ref{prop:linear-annhilators} we can now easily construct Koszul co-cycles of the form $m_1\wedge \cdots \wedge m_p\otimes f_{q,k,\bb}$. However, checking such a co-cycle is not a co-boundary is relatively difficult. Our key insight is that the issue of showing $m_1\wedge \cdots \wedge m_p\otimes f_{q,k,\bb}$ is not a co-boundary can, in a sense, be reduced to the special case considered in the Section~\ref{sec:special-case}. 

More precisely if we fix $0\leq q \leq |\nn|$ and $0\leq k \leq q$ so that $q-k\leq n_1$ and $k\leq n_2$ then the quotient map
\[
\begin{tikzcd}[column sep = 4em]
\overline{S}\rar{\pi}& \frac{\overline{S}}{\langle x_{q-k+1},x_{q-k+2},\ldots,x_{n_1},y_{k+1},y_{k+2},\ldots,y_{n_2}\rangle}=\overline{S}'
\end{tikzcd},
\]
induces a map between Koszul compexes
\begin{center}
\begin{tikzcd}[column sep = 3.5em, row sep = 3.5em]
\cdots \rar{}& \Alt^{p+1}\overline{S}_{\dd}\otimes \overline{S}_{(q-1)\dd+\bb}\dar[two heads]{\pi}\rar{\overline{\partial}_{p+1}}&\Alt^{p}\overline{S}_{\dd}\otimes \overline{S}_{q\dd+\bb}\dar[two heads]{\pi}\rar{\overline{\partial}_p}&\Alt^{p-1}\overline{S}_{\dd}\otimes \overline{S}_{(q+1)\dd+\bb}\dar[two heads]{\pi}\rar{}&\cdots \\
\cdots \rar{}& \Alt^{p+1}\overline{S}'_{\dd}\otimes \overline{S}'_{(q-1)\dd+\bb}\rar{\overline{\partial}'_{p+1}}&\Alt^{p}\overline{S}'_{\dd}\otimes \overline{S}'_{q\dd+\bb}\rar{\overline{\partial}'_p}&\Alt^{p-1}\overline{S}'_{\dd}\otimes \overline{S}'_{(q+1)\dd+\bb}\rar{}&\cdots 
\end{tikzcd}.
\end{center}
Checking directly in coordinates one sees that this induced map is in fact a map of chain complexes
\begin{align*}
\pi\left(\overline{\partial}_p(m_1\wedge \cdots\wedge m_{p}\otimes f)\right)&=\pi\left(\sum_{i=1}^p (-1)^{i}m_1\wedge \cdots \wedge\hat{m_i}\wedge \cdots\wedge m_{p}\otimes m_if\right)\\
&=\sum_{i=1}^p (-1)^{i}\pi(m_1)\wedge \cdots \wedge\hat{m_i}\wedge \cdots\wedge \pi(m_{p})\otimes \pi(m_if)\\
&=\sum_{i=1}^{p}(-1)^i \pi(m_1)\wedge \cdots \hat{\pi(m_i)}\wedge \cdots \wedge \pi(m_{p})\otimes \pi(m_i)\pi(f)) \\
&= \overline{\partial}'_p \left(\pi(m_1)\wedge \cdots\wedge \pi(m_{p})\otimes \pi(f)\right)= \overline{\partial}'_p\left(\pi(m_1\wedge \cdots\wedge m_{p}\otimes f)\right).
\end{align*}

Chasing this diagram of Koszul complexes shows that the condition of an element $\zeta\in \Alt^{p}\overline{S}_{\dd}\otimes \overline{S}_{q\dd+\bb}$ not being a co-boundary is implied by $\pi(\zeta)$ not being a co-boundary. 

\begin{lemma}\label{lem:technical-1}
Fix $\zeta\in \Alt^{p}\overline{S}_{\dd}\otimes \overline{S}_{q\dd}$. If $\pi(\zeta)\not\in \img\left(\overline{\partial}'_{p+1}\right)$ then $\zeta\not\in \img\left(\overline{\partial}_{p+1}\right)$.
\end{lemma}

\begin{proof}
Towards a contradiction suppose there exists $\alpha\in  \Alt^{p+1}\overline{S}_{\dd}\otimes \overline{S}_{(q-1)\dd+\bb}$ such that $\overline{\partial}_{p+1}(\alpha)=\zeta$. Now since $\pi$ induces a chain map of Koszul complexes
\[
\overline{\partial}'_{p+1}\left(\pi(\alpha)\right)=\pi\left(\overline{\partial}_{p+1}(\alpha)\right)=\pi(\zeta)
\]
contradicting the fact that $\pi(\zeta)\not\in \img\left(\overline{\partial}'_{p+1}\right)$.
\end{proof}

\begin{lemma}\label{lem:technical-2}
Fix $\zeta\in \Alt^{p}\overline{S}_{\dd}\otimes \overline{S}_{q\dd+\bb}$ and $m\in \overline{S}_{\dd}$. If $\zeta \not\in \img \overline{\partial}_{p+1,q}$ then $m\wedge \zeta\not\in \img \overline{\partial}_{p+2,q}$
\end{lemma}

\begin{proof}
We prove the contrapositive that if $m\wedge \zeta\in \img \overline{\partial}_{p+2,q}$ then $\zeta\in \img \overline{\partial}_{p+1,q}$. Towards this let $\zeta = \zeta_1\wedge \cdots \wedge \zeta_p\otimes f$, and assume that $\overline{\partial}_{p+2,q}(\alpha)=m\wedge \zeta$. Now we may write $\alpha$ as
\[
\alpha=m\wedge \left(\sum_{j} \xi_j\otimes g_j\right) + \sum_{i} \omega_i\otimes h_i
\]
where
\[
\xi_j\otimes g_j\in \Alt^{p}\Span_{\K}\left(\overline{S}_{\dd}-\{m\}\right)\otimes \overline{S}_{q\dd+\bb} \quad \quad \text{and} \quad \quad \omega_i\otimes h_i \in \Alt^{p+1}\Span_{\K}\left(\overline{S}_{\dd}-\{m\}\right)\otimes \overline{S}_{q\dd+\bb}.
\]
Computing we see that:
\begin{align}\label{eqn:reduction-2}
m\wedge \zeta = \overline{\partial}_{p+2,q}(\alpha) &= \overline{\partial}_{p+2,q}\left(m\wedge \left(\sum_{j} \xi_j\otimes g_j\right) + \sum_{i} \omega_i\otimes h_i\right)=\overline{\partial}_{p+2,q}\left(m\wedge \left(\sum_{j} \xi_j\otimes g_j\right)\right) + \overline{\partial}_{p+2,q}\left(\sum_{i} \omega_i\otimes h_i\right)\nonumber\\
&= \underbrace{-m\wedge \left(\overline{\partial}_{p+1,q}\left(\sum_{j} \xi_j\otimes g_j\right)\right)}_{\text{I}}\;\;+\;\;\underbrace{\sum_{j}-\xi_j \otimes mg_j+\overline{\partial}_{p+2,q}\left(\sum_{i} \omega_i\otimes h_i\right)}_{\text{II}}.
\end{align}
Now Part II of the above equation is entirely contained in the vector subspace $\Alt^{p}\Span_{\K}\left(\overline{S}_{\dd}-\{m\}\right)\otimes \overline{S}_{(q+1)\dd+\bb}$, and hence must equal zero. So Equation~\eqref{eqn:reduction-2} simplifies to:
\[
m\wedge \zeta =  -m\wedge \left(\overline{\partial}_{p+1,q}\left(\sum_{j} \xi_j\otimes g_j\right)\right).
\]
In particular, we see that $\overline{\partial}_{p+1,q}\left(-\sum_{j} \xi_j\otimes g_j\right)=\zeta$, and so as claimed $\zeta \in \img \overline{\partial}_{p+1,q}$.
\end{proof}

\begin{proof}[Proof of Theorem~\ref{thm:main2}]
By Proposition~\ref{prop:cohen-macaulay} the inequalities on $\bb$ and $\dd$ in the hypothesis of the theorem ensure that $S(\bb;\dd)$ is Cohen-Macaulay as an $R$-module. In particular, Proposition~\ref{prop:regular-sequence} implies that $\ell_{0},\ell_{1},\ldots,\ell_{|\nn|}$ is a linear regular sequence on $S(\bb;\dd)$, and so the Artinian reduction argument described in Corollary~\ref{cor:artinian-reduction} shows that quotienting by $\langle \ell_0,\ell_1,\ldots,\ell_{|\nn|}\rangle$ induces an isomorphism between $K_{p,q}(\nn,\bb;\dd)$ and $K_{p,q}^{\overline{R}}\left(\overline{S}(\bb;\dd)\right)$. 

Thus, it is enough to prove the desired non-vanishing for $K_{p,q}^{\overline{R}}\left(\overline{S}(\bb;\dd)\right)$. We do this by first using the special non-trivial syzygy on $\P^{q-k}\times\P^{k}$ constructed in Theorem~\ref{thm:special} together with the lifting argument in Lemma~\ref{lem:technical-1} to construct a single non-trivial syzygy on $\P^{\nn}$. We then construct other non-zero syzygies from this initial non-zero syzygy by Lemma~\ref{lem:technical-2}.

Set $\delta = r_{(q-k,k),\dd}-(q+1)$, and choose $\delta$ degree $\dd$ non-zero monomials $m_1,m_2,\ldots,m_{\delta}$ contained in the ideal  $\langle x_0,x_1,\ldots,x_{q-k},y_0,y_1,\ldots,y_k\rangle\overline{S} \cap \K[x_0,x_1\ldots,x_{q-k},y_0,y_1,\ldots,y_k]$. We wish to show that $\zeta = m_1\wedge \cdots \wedge m_{\delta} \otimes f_{q,k,\bb}$ represents a non-zero class in $K_{p,\delta}^{\overline{R}}\left(\overline{S}(\bb;\dd)\right)$. That is $\zeta$ is a represents a non-zero class in the cohomology of the following chain complex
\begin{equation}\label{eq:pf-chain}
\begin{tikzcd}[column sep = 3.5em, row sep = 3.5em]
\cdots \rar{}& \Alt^{\delta+1}\overline{S}_{\dd}\otimes \overline{S}_{(q-1)\dd+\bb}\rar{\overline{\partial}_{\delta+1}}&\Alt^{\delta}\overline{S}_{\dd}\otimes \overline{S}_{q\dd+\bb}\rar{\overline{\partial}'_\delta}&\Alt^{\delta-1}\overline{S}_{\dd}\otimes \overline{S}_{(q+1)\dd+\bb}\rar{}&\cdots. 
\end{tikzcd}
\end{equation}

Towards this we first show that $\zeta$ is well-defined and non-zero, which amounts to checking the same for $f_{q,k\bb}$. Since $b_1\geq0$ and $b_2\geq0$ we know by Lemma~\ref{lem-fqkb-well-defined} that $f_{q,k,\bb}$ is a well-defined monomial in $\overline{S}_{q\dd+\bb}$. Moreover, since $0\leq q-k+b_1<d_1$ and $0\leq k+b_2<d_2$ we know that $f_{q,k,\bb}$ is non-zero as an element of $\overline{S}$ (see Lemma~\ref{lem-fqkb-nonzero}).

Having showed that $\zeta$ is well-defined we turn to proving that $\zeta$ is not in the image $\overline{\partial}_{\delta+1}$. We do this by considering $\pi(\zeta)\in\Alt^{\delta}\overline{S}'_{\dd}\otimes \overline{S}'_{q\dd+\bb}$ where $\pi$ is as defined in the beginning of this section. Using the inductive structure described in Lemma~\ref{lem:induction-ideal} we know that $\overline{S}'$ is exactly $\overline{S}$ in the case when $n_1=q-k$ and $n_2=k$. Thus, the Koszul complex 
\begin{equation}\label{eq:reduction}
\begin{tikzcd}[column sep = 3.5em, row sep = 3.5em]
\cdots \rar{}& \Alt^{\delta+1}\overline{S}'_{\dd}\otimes \overline{S}'_{(q-1)\dd+\bb}\rar{\overline{\partial}'_{\delta+1}}&\Alt^{\delta}\overline{S}'_{\dd}\otimes \overline{S}'_{q\dd+\bb}\rar{\overline{\partial}'_\delta}&\Alt^{\delta-1}\overline{S}'_{\dd}\otimes \overline{S}'_{(q+1)\dd+\bb}\rar{}&\cdots 
\end{tikzcd}
\end{equation}
actually computes $K_{\delta,q}((q-k,k),\bb;\dd)$. Moreover, by construction one sees that $\pi(\zeta)$ represents one of the non-trivial syzygies constructed in Theorem~\ref{thm:special}. In particular, $\pi(\zeta)$ represents a non-zero element in the cohomology of complex~\eqref{eq:reduction} above. This means that
$\pi(\zeta)$ is not in the image of $\overline{\partial}'_{\delta+1}$. 

Now by Lemma~\ref{lem:technical-1} the fact that $\pi(\zeta)$ is not in the image of $\overline{\partial}'_{\delta+1}$ implies that $\zeta\in \Alt^{\delta}\overline{S}_{\dd}\otimes \overline{S}_{q\dd+\bb}$ is not in the image of $\overline{\partial}_{\delta+1}$. Thus, to show that $\zeta$ is a non-trivial syzygy on $\P^{\nn}$, i.e. a non-zero element of the cohomology of complex \eqref{eq:pf-chain} above, we must show that $\zeta$ is in the kernel of $\overline{\partial}_{\delta}$. Using our description of the annihilators of $f_{q,k,\bb}$ given in Proposition~\ref{prop:linear-annhilators} we know that $m_1,m_2,\ldots,m_{\delta}$ annihilate $f_{q,k,\bb}$. So by part (1) of Proposition~\ref{prop:nonvanishing-bounds} implies that $\overline{\partial}_{\delta}(\zeta)=0$. Hence $\zeta$ represents a non-trivial class in $K_{p,\delta}^{\overline{R}}\left(\overline{S}(\bb;\dd)\right)$.  

We now use Lemma~\ref{lem:technical-2} to construct other non-trivial syzygies from $\zeta$. In particular, by inductively applying Lemma~\ref{lem:technical-2} we know that if $(n_1\wedge\cdots\wedge n_t)\wedge \zeta$ is non-zero then $(n_1\wedge\cdots\wedge n_t)\wedge \zeta$ is not in the image of $\overline{\partial}_{\delta+t+1,q}$. Thus, as long as $(n_1\wedge\cdots\wedge n_t)\wedge \zeta$ remains non-zero and in the kernel of $\overline{\partial}_{\delta+t,q}$ it will represent a non-trivial class in $K_{\delta+t,q}^{\overline{R}}(\overline{S}(\bb;\dd))$. 

Using the description of the annihilators of $f_{q,k,\bb}$ given in Proposition~\ref{prop:linear-annhilators} together with part (1) of Proposition~\ref{prop:nonvanishing-bounds} we know that $(n_1\wedge\cdots\wedge n_t)\wedge \zeta$ will be in the kernel of $\overline{\partial}_{\delta+t,q}$ so long as $n_i\in \langle {x}_{0},{x}_1,\ldots,{x}_{q-k-1},{y}_0,{y}_1,\ldots,{y}_{k-1}\rangle\overline{S}$ for all $i$. Further, $(n_1\wedge\cdots\wedge n_t)\wedge \zeta$ will be non-zero provided that $n_1,n_2,\ldots,n_t,m_1,m_2,\ldots,m_{\delta}$ are unique in $\overline{S}$. As these are all monomials of bi-degree $\dd$ contained in $\langle {x}_{0},{x}_1,\ldots,{x}_{q-k-1},{y}_0,{y}_1,\ldots,{y}_{k-1}\rangle\overline{S}$ the number of such elements is controlled by the Hilbert function of this ideal. Using Lemma~\ref{lem:hilbert-function-bound} to compute the Hilbert function of $\langle {x}_{0},{x}_1,\ldots,{x}_{q-k-1},{y}_0,{y}_1,\ldots,{y}_{k-1}\rangle\overline{S}$ we see that we can construct a non-trivial class in $K_{\delta+t,q}^{\overline{R}}(\overline{S}(\bb;\dd))$ whenever
\begin{align*}
\delta+t\leq \HF\left(\dd, \langle {x}_{0},{x}_1,\ldots,{x}_{q-k-1},{y}_0,{y}_1,\ldots,{y}_{k-1}\rangle\overline{S}\right)&\geq \HF\left(\dd,\langle {x}_{0},{x}_1,\ldots,{x}_{q-k-1},{y}_0,{y}_1,\ldots,{y}_{k-1}\rangle S\right)-\HF\left(\dd,\R(\nn,\dd)\right)\\
&=r_{\nn,\dd}-\binom{d_1+n_1-(q-k)}{n_1-(q-k)}\binom{d_2+n_2-k}{n_2-k}-(|\nn|+1).
\end{align*}
\end{proof}

\begin{proof}[Proof of Theorem~\ref{thm:main}]
This follows immediately from Theorem~\ref{thm:main2} with $\bb=\zero$.
\end{proof}

\begin{proof}[Proof of Corollary~\ref{cor:main}]
By Theorem~\ref{thm:main} if $d_1>q$ and $d_2>q$ then 
\begin{align*}
\rho_q(\nn;\dd)\geq&1-\frac{\min\left\{\displaystyle \binom{d_1+n_1-i}{n_1-i}\binom{d_2+n_2-j}{n_2-j} \;\;\; \bigg| \;\;\; \begin{matrix} i+j=q \\ 0\leq i \leq n_1 \\ 0\leq j \leq n_2 \end{matrix}\right\}}{r_{\nn,\dd}}-\frac{\min\left\{\displaystyle \binom{d_1+i}{i}\binom{d_2+j}{j} \;\;\; \bigg| \;\;\; \begin{matrix}i+j=q \\ 0\leq i \leq n_1 \\ 0\leq j \leq n_2 \end{matrix}\right\}}{r_{\nn,\dd}} - \frac{|\nn|-q-1}{r_{\nn,\dd}}\\
\\
\geq&1-\sum_{\substack{i+j=q \\ 0\leq i \leq n_1 \\ 0\leq j \leq n_2}} \frac{\displaystyle \binom{d_1+n_1-i}{n_1-i}\binom{d_2+n_2-j}{n_2-j}}{r_{\nn,\dd}}+\frac{\displaystyle \binom{d_1+i}{i}\binom{d_2+j}{j}}{r_{\nn,\dd}}- \frac{|\nn|-q-1}{r_{\nn,\dd}}.
\end{align*}
The result follows by noting that $\binom{d+n}{n}=\frac{d^n}{n!}+O(d^{n-1})$ and $r_{\nn,\dd}=O(d_1^{n_1}d_2^{n_2})$.
\end{proof}

%%%%%%%%%%%%%%%%%%%%%%%%%%%%%%%%%%%%%%%%%%%%%%%%%%%%%%%%%%%%%%%%%%%%%%%%%%%%%%%%%%%%%%%%%%%%%%%%%%%%%%%%%%%%%%%%

\end{document}